\documentclass[10pt]{amsart}
\usepackage[a4paper, margin=3cm]{geometry} 

\usepackage{amsmath,amssymb,amscd,amsthm,amsxtra}
\usepackage{dsfont}
\usepackage{graphicx, mathrsfs}
\usepackage{amssymb,amsmath,amsthm}
\usepackage{mathrsfs,dsfont}
\usepackage{fancyhdr}
\usepackage{enumerate}
\usepackage{hyperref}
\usepackage[toc,page]{appendix}
\usepackage{color}
\usepackage{quiver}
\usepackage{comment}
\usepackage[dvipsnames]{xcolor}
\usepackage[T1]{fontenc}
\usepackage[utf8]{inputenc}
\newtheorem{thm}{Theorem}[section]
\newtheorem{prop}[thm]{Proposition}
\newtheorem{lem}[thm]{Lemma}

\newtheorem{question}[thm]{Question} 
\theoremstyle{definition}
\newtheorem{definition}[thm]{Definition}
\newtheorem{example}[thm]{Example}
\newtheorem{lemma}[thm]{Lemma}

\newtheorem{remark}[thm]{Remark}
\numberwithin{equation}{section}

\newcommand{\Z}{\mathbb{Z}}
\newcommand{\R}{\mathbb{R}}
\newcommand{\T}{\mathbb{T}}

\newcommand{\exist}{\exists}

\newcommand{\sub}{\subset}

\newcommand{\rmG}{\mathrm{G}}
\newcommand{\rmd}{\mathrm{d}}

\newcommand{\Span}{\mathrm{Span}}
\newcommand{\vol}{\mathrm{vol}}
\newcommand{\SO}{\mathrm{SO}}
\newcommand{\Sp}{\mathrm{Sp}}
\newcommand{\SU}{\mathrm{SU}}

\newcommand{\abs}[1]{\left| #1 \right|}

\newcommand{\adj}{\mathrm{adj}}
\newcommand{\Iso}{\mathrm{Iso}}
\newcommand{\frt}{\mathfrak{t}}
\newcommand{\friso}{\mathfrak{iso}}
\newcommand{\Sym}{\mathrm{Sym}}

\newcommand{\Ric}{\mathrm{Ric}}

\DeclareMathOperator{\Hol}{Hol}

\DeclareMathOperator{\Null}{Null}

\begin{document}


\title{Closed $\rmG_2$-structures with $\mathbb{T}^3$-symmetry and hypersymplectic structures}


\author{Chengjian Yao, Ziyi Zhou}
\address{Institute of Mathematical Sciences, ShanghaiTech University}
\address{393 Middle Huaxia Rd, Pudong New District, Shanghai, 201210}

\email{yaochj@shanghaitech.edu.cn, zhouzy22023@shanghaitech.edu.cn}


\begin{abstract}We decompose linear $\rmG_2$-structure in canonical ways adapted to 3-dimensional subspaces, in terms of certain natural 1-forms and definite triple of 2-forms, and apply the decompositions to the study of $\rmG_2$-structure with $\T^3$-symmetry. Closed $\rmG_2$-structures $\varphi$ with an effective $\mathbb{T}^3$-symmetry on connected manifolds are roughly classified into two types according the orbits being non-isotropic or isotropic. Type I: if some orbit is non-isotropic, then the action is almost-free and $\varphi$ reduces to a \emph{good hypersymplectic orbifold} with cyclic isotropic groups. Type II: if some orbit is isotropic, then the action is \emph{locally multi-Hamiltonian} for $\varphi$. Moreover, the open and dense subset of principal orbits is foliated by $\mathbb{T}^3$-invariant hypersymplectic manifolds.

    If $\varphi$ is torsion-free, then for Type I, there arises another natural hypersymplectic structure, and a generalized Gibbons-Hawking Ansatz extending \emph{Madsen-Swann Ansatz} is derived. For Type II, $\varphi$ is \emph{locally toric}.

    Assuming moreover completeness and constant orbit volume, exactly three possibilities occur. Type Ia: orbits are purely non-isotropic non-associative, then the hypersymplectic 4-orbifold becomes a flat manifold. Type Ib: orbits are purely associative, then the $\T^3$-action is flat, and the hypersymplectic 4-orbifold becomes a hyperk\"ahler 4-orbifold. Type II: orbits are isotropic, then all orbits are principal, and $\varphi$ is flat.
    
 
\end{abstract}
\maketitle
\noindent {\bfseries{Keywords}}: Closed $\rmG_2$-structure, Hypersymplectic structure, Torsion-free $\rmG_2$-structure, Hyperk\"ahler manifold, $\rmG_2$-Laplacian flow.


\section{Introduction and preliminaries}
The group $\rmG_2$ is one of the exceptional simple Lie groups, and the $\rmG_2$ geometry has become more and more important. The basics of $\rmG_2$ geometry could be found in the excellent survey \cite{Kar} by Karigiannis. We recall very briefly below.
\begin{definition}
    $\rmG_2$ is the subgroup of $\SO(7;\R)$ preserving the $3$-form
    \[
        \varphi_o:=e^{123}-e^1(e^{45}+e^{67})-e^2(e^{46}+e^{75})-e^3(e^{47}+e^{56})\in\Lambda^3(\R^7)^*,
    \]
    where $e^1,\cdots,e^7$ is the standard basis of $(\R^7)^*$ and $e^{ijk}:=e^i\wedge e^j\wedge e^k$.
\end{definition}
\begin{definition}
    Let $M^7$ be a smooth manifold. A \emph{$\rmG_2$-structure} on $M$ is a smooth $3$-form $\varphi$ on $M$ such that for any $p\in M$, there exists a linear isomorphism $T_pM\cong\R^7$ such that $\varphi_p\in\Lambda^3(T^*_pM)$ corresponds to $\varphi_o\in\Lambda^3(\R^7)^*$. The corresponding basis of $T_pM$ is called an \emph{adapted basis}.
\end{definition}

A smooth manifold $M^7$ admits a $\rmG_2$-structure if and only if $M$ is both orientable and spinnable, and a $\rmG_2$-structure $\varphi$ determines a Riemannian metric $g_\varphi$ via the formula
    \begin{align*} 
    g_\varphi(u,v)\text{dvol}_{g_\varphi}=\frac{1}{6}u\lrcorner\varphi\wedge v\lrcorner\varphi\wedge\varphi, \;\; \forall u,v\in T_pM,\forall p\in M.
    \end{align*}

The $\rmG_2$-structure is called \emph{closed} if $\rmd\varphi=0$, and is called \emph{coclosed} if $\rmd^{*_\varphi}\varphi=0$. The structure $\varphi$ is \emph{torsion-free}, i.e. $\nabla\varphi=0$ if and only if $\varphi$ is closed and coclosed where $\nabla$ is the Levi-Civita connection of the metric $g_\varphi$. The pair $(M^7,\varphi)$ of smooth manifold together with a torsion-free $\rmG_2$-structure is called a \emph{$\rmG_2$-manifold}. A key property in the study of $\rmG_2$ geometry is that $\Hol(g_\varphi)\subset\rmG_2$ if and only if $\left(M, \varphi\right)$ is a $\rmG_2$-manifold. If this happens, the Riemannian metric $g_\varphi$ is Ricci-flat. The general existence, uniqueness and moduli space of torsion-free $\rmG_2$-structures remain mysterious, despite much progress \cite{Kar}.\\

The first and main motivation of the current study comes from the curiosity about all possible reduction of closed $\rmG_2$-structure with $\mathbb{T}^3$-symmetry. On the one hand, it is shown by Podest\`a-Raffero \cite{PR} that the identity component of the automorphism group of a closed $\rmG_2$-structure on a compact manifold must be Abelian. On the other hand, the recent study of $\rmG_2$-Laplacian flow \cite{LW1, LW2} and its several dimensional reductions \cite{FR,PS}, in particular the hypersymplectic flow \cite{Fine-Yao, Fine-Yao2} poses the natural question about the relation between closed $\rmG_2$-structure with $\T^3$-symmetry and hypersymplectic structures \cite{Don, Fine-Yao}. Due to the extreme complexity of the analysis about the $\rmG_2$-Laplacian flow, it is tempting to extend the very $\rmG_2$-Laplacian flow coming out of hypersymplectic flow \cite{Fine-Yao} to the general $\T^3$-symmetric situation, where potentially more geometric quantities make it plausible to expect conditional long time existence of the flow. We recall the definition of hypersymplectic structure.

\begin{definition}
    A \emph{hypersymplectic structure} on a $4$-manifold $N$ is a triple $\underline\omega=\left(\omega_1, \omega_2, \omega_3\right)$ of symplectic forms on $N$ such that $\left(\frac{\omega_i\wedge\omega_j}{2\mu}\right)>0$ on $N$ for some volume form $\mu$.
\end{definition}
Similar to how $\rmG_2$-structure gives rise to a Riemannian metric, hypersymplectic structure $\underline\omega$ determines a Riemannian metric via the formula
\begin{align}
\label{formula:definite}
g_{\underline\omega}(u,v)\text{dvol}_{g_{\underline\omega}}
= 
\frac{1}{6}\varepsilon^{ijk}u\lrcorner \omega_i \wedge v\lrcorner \omega_j \wedge \omega_k, \;\; \forall u,v\in T_pN,\forall p\in N.
\end{align}
Actually, that the formula determines a Riemannian metric does not require the \emph{closedness} of $\omega_i$'s, and this determination works whenever the matrix $\left(\frac{\omega_i\wedge\omega_j}{2\mu}\right)$ is pointwise definite, which we call a \emph{definite triple}. The notion of hypersymplectic structure leads to a nice characterization of hyperk\"ahler structure, that is the Riemannian metric $g_{\underline\omega}$ is hyperk\"ahler if and only if the matrix $\left(\frac{\omega_i\wedge\omega_j}{2\text{dvol}_{g_{\underline\omega}}}\right)$ is constant. The relation between a hypersymplectic structure and closed $\rmG_2$-structure is given in Example \ref{example-standard}.\\

The second motivation is to understand $\rmG_2$-manifold with $\T^3$-symmetry, in particular along the line of study of ``toric $\rmG_2$-manifolds'' by Madsen-Swann \cite{Madsen-Swann} using \emph{multi-moment map}. Let us first recall the notion of multi-moment map from \cite{Madsen-Swann}.

\begin{definition}[multi-Hamiltonian/locally multi-Hamiltonian]
    Let $M$ be a manifold, and $\alpha$ be a closed $(k+1)$-form on $M$. Let $\mathrm G$ be an Abelian Lie group action on $M$ preserving $\alpha$. A \emph{multi-moment map} for $\alpha$ is an invariant map $\nu:M\to\Lambda^k\mathfrak g^*$ such that
    \[
    \mathrm d\langle \nu,W\rangle=\xi(W)\lrcorner\alpha, 
    \]
    for all $W\in\Lambda^k\mathfrak g$ where $\xi(W)$ is the unique multivector determined by $W$ via the action. The action is called \emph{multi-Hamiltonian} for $\alpha$ if there is a multi-moment map for $\alpha$. If it only holds on a $\rmG$-invariant open neighborhood near any $\rmG$-orbit, it is called \emph{locally multi-Hamiltonian}.
\end{definition}

Madsen and Swann initiated the study of $\rmG_2$-manifold with multi-Hamiltonian torus actions in \cite{Madsen-Swann}, in particular $\mathbb{T}^3$-action. In order to get a full picture of $\rmG_2$-manifold with $\T^3$-symmetry, we extend the notion of toric $\rmG_2$-manifold to \emph{locally toric} $\rmG_2$-manifold. 
\begin{definition}[Toric/Locally toric $\rmG_2$-manifold \cite{Madsen-Swann}]
    A $\rmG_2$-manifold $(M,\varphi)$ with an effective $\mathbb T^3$-action is called \emph{toric} if it admits multi-moment maps $\nu:M\to\Lambda^2(\mathfrak t^3)^*\simeq\R^3$ for $\varphi$ and $\mu:M\to\Lambda^3(\mathfrak t^3)^*\simeq\R$ for $*\varphi$. If it only holds on a $\mathbb{T}^3$-invariant open neighborhood near any point, it is called \emph{locally toric}.
\end{definition}

The toric condition—even in its local form—is natural from the viewpoint of multi-moment maps, but it imposes a strong additional constraint beyond mere $\T^3$-symmetry. As we shall see in Section \ref{sect:isotropic}, a $\T^3$-invariant closed $\rmG_2$-structure with isotropic orbits is automatically locally toric. However, this is far from being the generic situation: a generic $\T^3$-invariant $\rmG_2$-structure does not possess a multi-moment map for $\varphi$, and therefore falls outside the toric (and even locally toric) framework altogether. One of the motivations for the present article is precisely to understand what structures replace the multi-moment maps when the toric condition is relaxed.

The following examples illustrate the range of possibilities. Example \ref{example-standard} (1) exhibits all three orbit types—associative, isotropic, and neither isotropic nor associative — within a single flat model, and shows that isotropic orbits are locally toric but not globally toric. Example \ref{example-standard} (2) provides a large class of non-flat, closed $\rmG_2$-structure whose orbits are all associative. Together, these examples demonstrate that (locally) toric $\rmG_2$-manifolds occupy a special corner of the landscape of $\T^3$-invariant closed $\rmG_2$-structures, and that a systematic study beyond the toric realm is necessary.
\vspace{0.3cm}

\begin{example}~
\label{example-standard}
\begin{enumerate}
    \item Let $\left(\R^7/\Gamma, \varphi_o\right)$ where $\Gamma\subset \R^7$ be a lattice, and $\Pi\subset \R^7$ be any ``rational'' $3$-subspace, i.e. the subspace spanned (over $\R$) by vectors in $\Gamma$. Then there is an induced free $\T^3$-action action $\rho_\Pi$. More precisely, let $\left\{U_1, U_2, U_3\right\}\subset \Gamma$ be a primitive basis of $\Pi$, then 
    \begin{align*}
        \rho_\Pi: \mathbb{T}^3
        &:=\R^3/\left(2\pi\mathbb{Z}\right)^3\rightarrow \text{Diff}\left( \R^7/\Gamma, \varphi_o\right),\\
        \rho_\Pi\left(\underline\lambda\right)\cdot \underline x
        & = 
        \sum_{i=1}^3\frac{\lambda_i}{2\pi} U_i +\underline x.
    \end{align*} 
       \begin{itemize}
           \item If $\Pi$ is associative, then each orbit of the action $\rho_\Pi$ is associative;
           \item If $\Pi$ is isotropic, then each orbit of the action $\rho_\Pi$ is isotropic, and it is locally toric but not toric;
           \item For generic $\Pi$, the orbit is neither associative nor isotropic.
       \end{itemize}
    \item Let $\left(X, \underline\omega=\left(\omega_1,\omega_2,\omega_3\right)\right)$ be a hypersymplectic $4$-manifold. Then $M:=\T^3\times X$ is equipped with a closed $\rmG_2$-structure
    \begin{align*}
    \varphi&=\rmd t^{123}-\rmd t^1\wedge\omega_1-\rmd t^2\wedge\omega_2-\rmd t^3\wedge\omega_3
    \end{align*}
    whose dual $4$-form is 
    \begin{align*}
        *\varphi&=\text{dvol}_{g_{\underline\omega}}-\rmd t^{23}\wedge\tilde\omega^1-\rmd t^{31}\wedge\tilde\omega^2-\rmd t^{12}\wedge\tilde\omega^3,
    \end{align*}
    where $\T^3=S^1\times S^1\times S^1=\left\{(e^{it_1},e^{it_2},e^{it_3})|t_1,t_2,t_3\in [0,2\pi]\right\}$, $\tilde\omega^i=Q^{ij}\omega_j$, $Q_{ij}=\frac{1}{2}\langle \omega_i,\omega_j\rangle_{g_{\underline\omega}}$ and $\left(Q^{ij}\right)=\left(Q_{ij}\right)^{-1}$ (see \cite{Fine-Yao} for more details). The canonical $\mathbb{T}^3$-action, which acts as translations in the first factor of $\mathbb{T}^3\times X$, is locally multi-Hamiltonian for $\varphi$ and is free with each orbit being associative. If $\left(X, \underline\omega\right)$ is locally tri-Hamiltonian hyperk\"ahler, then $\left(M, \varphi\right)$ is locally toric $\rmG_2$-manifold. As mentioned above, this is a rather restricted class of $\T^3$-invariant closed $\rmG_2$-structure.
\end{enumerate}
\end{example}

\begin{remark}    
    There are infinitely many known very interesting and geometrically rich complete toric $\rmG_2$-manifolds with full $\rmG_2$-holonomy \cite{DMS, FHN}. However, without the assumption of completeness there do exist examples of $\rmG_2$-manifolds with full $\rmG_2$-holonomy and $\T^3$-symmetry, given by torsion-free hypersymplectic structures \cite{Don3, Fine-Yao2} through the way in Example \ref{example-standard}. Therefore, it is interesting to ask if complete $\rmG_2$-manifold with $\T^3$-symmetry and associative orbits can have full holonomy.
\end{remark}

With the mentioned motivations, we consider closed $\rmG_2$-structure and then torsion-free $\rmG_2$-structures  with effective $\mathbb{T}^3$-symmetry in this article. Compactness of the manifold is not assumed in general, but we do assume the 7-manifold is connected throughout the article. Let $(M^7,\varphi)$ be a $7$-manifold with a $\T^3$-invariant closed $\rmG_2$-structure, and the action is effective. Let $\xi:\frt^3\to\friso(M,\varphi)\sub\Gamma(TM)$ be the differential of the action $\T^3\to\Iso(M,\varphi)$. Take a basis $\left\{\mathfrak{u}_1, \mathfrak{u}_2, \mathfrak{u}_3\right\}$ for $\frt^3$, and let $U_i:=\xi(\mathfrak{u}_i)$ be the vector field on $M$ generated by $\mathfrak{u}_i$. Then by the Cartan formula,
\[
U_i\lrcorner\varphi,\;
U_i\wedge U_j\lrcorner\varphi
\]
are closed $\T^3$-invariant forms and $\varphi(U_1,U_2,U_3)$ is a constant function. If this constant vanishes, then each orbit is isotropic. If it is nonzero, then each orbit is non-isotropic. If further $\varphi$ is assumed to be torsion-free, then 
\begin{align*}
U_1\wedge U_2\wedge U_3\lrcorner{*\varphi}
\end{align*}
is also closed and invariant.\\

We summarize the structure and main results of this paper. Troughout the article, we assume the 7-manifold in consideration is connected. As a preparation, in section \ref{sect:canonical-decomposition} we decompose a linear $\rmG_2$-form canonically adapted to a 3-dimensional subspace, in terms of certain natural 1-forms and definite triples. The results are presented here as elementary linear algebra lemmas (Lemma \ref{lem:canonical-decomposition-isotropic} and Lemma \ref{lem:canonical-non-isotropic}), to emphasize the natural emerging of definite triples of 2-forms, and to make them applicable to more general scenarios such as $\rmG_2$-structures with associative fibrations or $3$-dimensional symmetry. In the following sections, we apply the canonical decompositions to the concrete situation of $\rmG_2$-structure with effective $\T^3$-symmetry. Firstly, closed $\rmG_2$-structures with effective $\mathbb{T}^3$-symmetry are roughly classified into the following two types, according to the orbits being non-isotropic or isotropic:
\begin{enumerate}
    \item[1)] Type I: orbits are non-isotropic. In this case, the action is almost-free, and $M/\mathbb{T}^3$ is a \emph{good hypersymplectic orbifold} with singularities of cyclic type, and $\left(M,\varphi\right)$ is covered by $\left(\T^3\times \boldsymbol X, c\rmd t^{123}-c^{-1}\rmd t^i\wedge \boldsymbol\omega_i +\boldsymbol\Omega\right)$ where $\underline{\boldsymbol\omega}=(\boldsymbol\omega_1,\boldsymbol\omega_2,\boldsymbol\omega_3)$ is a hypersymplectic structure, and $\boldsymbol \Omega$ is a closed $3$-form on $\boldsymbol X$. Moreover, $\boldsymbol X\rightarrow M/T^3$ is an orbifold covering and $\underline{\boldsymbol\omega}, \boldsymbol\Omega$ descend to the 4-orbifold (see Theorem \ref{thm:nonisotropic-closed-structure}).
    \item[2)] Type II: orbits are isotropic. In this case, the action is locally multi-Hamiltonian for $\varphi$, and the open dense subset of principal orbits is foliated by hypersymplectic manifolds with free $\mathbb{T}^3$-symmetry. See section \ref{sect:hypersymplectic-foliation} for detail.
\end{enumerate}

\noindent Secondly, if the closed $\rmG_2$-structure is assumed to be torsion-free, we obtain some partial classification results where hypersymplectic geometry plays an essential role:
\begin{enumerate}
    \item[1)] Type I: orbits are non-isotropic. Two hypersymplectic structures $\underline\omega=(\omega_1,\omega_2,\omega_3)$ and $\underline{\tilde\omega}=(\tilde\omega^1,\tilde\omega^2,\tilde\omega^3)$ are singled out from $\varphi$ and $*\varphi$ (See Theorem \ref{thm:nonisotropic-torsionfree-structure}) respectively. Also, a Generalized Gibbons–Hawking Ansatz analogous to Madsen-Swann Ansatz \cite{Madsen-Swann} is obtained (See Theorem \ref{thm:Gibbons-Hawking}).
    \item[2)] Type II: orbits are isotropic. The $\rmG_2$-structure is \emph{locally toric}. See section \ref{sect:trivalent-graph} for detail.
\end{enumerate}

\noindent Thirdly, if the torsion-free $\rmG_2$-structure is further assumed to be complete with constant volume of orbits, we prove several Liouville type theorems addressing questions from the second motivation above. Exactly three possibilities occur:
\begin{enumerate}
    \item[1.1)] Type Ia: orbits are purely non-isotropic non-associative. Then the 4-orbifold becomes a flat manifold, and $\varphi$ is flat (See Theorem \ref{thm:complete-torsion-free-constant-volume-non-associative}).
    \item[1.2)] Type Ib: orbits are purely associative. Then the hypersymplectic 4-orbifold becomes a hyperk\"ahler orbifold, and the $\T^3$-action is flat (See Theorem \ref{thm:associative-hyperkahler}).
    \item[2)] Type II: orbits are isotropic. Then all orbits are principal, and $\varphi$ is flat (See Theorem \ref{thm:complete-isotropic-const-volume}).
\end{enumerate}

The geometric assumption on the volume of the orbits appears to be technical at this moment as we cannot rule out the co-existence of non-associative and associative orbits. As Example \ref{ex:coexistence} shows, non-associative and associative orbits could indeed coexist for $\T^3$-invariant closed $\rmG_2$-structure. However, we expect it possible to be removed for torsion-free $\rmG_2$-structure at least in the case of non-isotropic orbits. We should also remark that the various classification results presented in this article are limited to manifold with holonomy group strictly contained in $\rmG_2$, and are thus Liouville type theorems. The full classification of complete (locally) toric $\rmG_2$-manifolds, the class of which most of the known full $\rmG_2$-holonomy manifolds with $\T^3$-symmetry belong to, to the extend of $S^1$-invariant complete hyperk\"ahler 4-manifold \cite{Swann} seems very challenging.

\section{Canonical decomposition adapted to 3-dimensional subspaces}
\label{sect:canonical-decomposition}

Let $\varphi$ be a linear $\rmG_2$-structure on the $7$-dimensional real vector space $\mathbb{V}$, i.e. a constant coefficients $\rmG_2$-structure on the manifold $\mathbb{V}$ under any fixed global basis of $\mathbb{V}$. Recall that a subspace $\mathbb{W}\subset \mathbb{V}$ is called \emph{isotropic} if $\varphi|_\mathbb{W}=0$, and \emph{non-isotropic} otherwise. The structure $\varphi$ determines a cross-product $\times_\varphi$ and a Riemannian metric $g_\varphi$, and a 3-dimensional subspace $\mathbb{W}$ is called \emph{associative} if $\mathbb{W}$ is closed under $\times_\varphi$, and \emph{non-associative} otherwise. The orthogonal complement of an associative subspace is called a \emph{coassociative} subspace.

Given any $3$-dimensional subspace $\mathbb{U}$, we are going to decompose $\varphi$ in terms of the natural $1$-forms, i.e. elements in 
\[
\Lambda_\mathbb{U}= \left\{\left(U'\wedge U''\right)\lrcorner \varphi\in\mathbb{V}^*|U', U''\in \mathbb{U}\right\},
\]
and certain definite triple of $2$-forms on the common kernel space of these $1$-forms, i.e. 
\[
\mathbb{K}_\mathbb{U}:=\Null\left(\Lambda_\mathbb{U}\right)= \bigcap_{\alpha\in \Lambda_\mathbb{U}\backslash\{0\}} \ker\alpha.
\]
The subset $\Lambda_\mathbb{U}\subset \mathbb{V}^*$ is a $3$-dimensional subspace, and $\mathbb{K}_\mathbb{U}\subset \mathbb{V}$ is a $4$-dimensional subspace. The relation between $\mathbb{U}$ and $\mathbb{K}_\mathbb{U}$ in the following typical situations is:
\begin{itemize}
    \item If $\mathbb{U}$ is isotropic, then $\mathbb{U}\subset \mathbb{K}_{\mathbb{U}}$;
    \item If $\mathbb{U}$ is non-isotropic and non-associative, then $\mathbb{U}$ is transversal to $\mathbb{K}_{\mathbb{U}}$, and the orthogonal projection induces an isomorphism $pr:\mathbb{K}_{\mathbb{U}}\rightarrow \mathbb{U}^{\perp_{g_\varphi}}$;
    \item If $\mathbb{U}$ is associative, then $\mathbb{K}_{\mathbb{U}}=\mathbb{U}^{\perp_{g_\varphi}}$ and $pr=id$.
\end{itemize}

In this section, we derive several elementary linear algebra lemmas describing the relation between $3$-dimensional subspaces and definite triples of $2$-forms for a $\rmG_2$-structure. These lemmas will be useful in the study of $\rmG_2$-structures with $3$-dimensional symmetries (such as $\T^3$ or $\SU(2)$ symmetry), and manifolds with fibration structures (such as $\rmG_2$-manifolds fibered by associative or isotropic submanifolds) in a way analogous to Karigiannis-Lotay's treatment on coassociative fibrations \cite{Kari-Lotay}. Let us emphasize that the decomposition in the isotropic case below is obtained in Madsen-Swann \cite{Madsen-Swann}, and we extend their arguments and detail out the computations in order to see how definite triples of 2-forms emerge from a $\rmG_2$-structure, which is a crucial new ingredient for our study. 

\begin{definition}[Horizontal form]
    A $k$-form $\lambda$ on $\mathbb{V}$ is called horizontal if $U\lrcorner\lambda =0$ as a $(k-1)$-form for any $U\in \mathbb{U}$.
\end{definition}

\subsection{Non-associative subspace}\,
Suppose $\mathbb{U}$ is non-associative $3$-dimensional subspace. Let $\left\{e_1, e_2, e_3\right\}$ be an orthonormal basis of $\mathbb{U}$. The ``non-associativity'' means $e_3$ and $e_1\times e_2$ are not proportional, which in turns implies the existence of the unit vector
\[
\tilde e_3 
:= 
\frac{e_3-g_\varphi\left(e_3, e_1\times e_2\right)e_1\times e_2}{\left|e_3-g_\varphi\left(e_3, e_1\times e_2\right)e_1\times e_2\right|}
\]
which is orthogonal to $e_1, e_2, e_1\times e_2$. The key feature of the ``non-associativity'' of $\mathbb{U}$ is that the three vectors $e_1, e_2, e_3$ span $\mathbb{V}$ under the cross product $\times=\times_\varphi$. Define
\[
\tilde e_1:=e_1,\;\tilde e_2:=e_2,
\]
then we have an adapted basis \cite{Madsen-Swann}
\begin{equation}
\label{adapted}
    -\tilde e_2\times\tilde e_3,\;
    -\tilde e_3\times\tilde e_1,\;
    -\tilde e_1\times\tilde e_2,\;
    -[\tilde e_1, \tilde e_2,\tilde e_3],\;
    \tilde e_1,\;
    \tilde e_2,\;
    \tilde e_3,
\end{equation}
where $\tilde e_i\times\tilde e_j:=\left(\tilde e_{ij}\lrcorner\varphi\right)^\sharp,\left[\tilde e_1,\tilde e_2,\tilde e_3\right]:=\left(\tilde e_{123}\lrcorner{*\varphi}\right)^\sharp$.
Under this adapted basis, we can write (for simplicity we ignore the symbol ``$\wedge$'' and ``$\otimes$'' between covectors if there is no confusion)
\begin{align*}
    \varphi
    &=-\tilde\alpha^{123}-\tilde\alpha^i\tilde\beta\tilde e^i+\frac12{\varepsilon}_{ijk}\tilde\alpha^i\tilde e^{jk},\\
    *\varphi
    &=\tilde e^{123}\tilde\beta+\frac12{\varepsilon_{ijk}}\tilde\alpha^{ij}\tilde\beta\tilde e^k-\frac12\tilde\alpha^{ij}\tilde e^{ij},\\
    g_\varphi
    &=(\tilde e^1)^2+(\tilde e^2)^2+(\tilde e^3)^2+(\tilde\alpha^1)^2+(\tilde\alpha^2)^2+(\tilde\alpha^3)^2+\tilde\beta^2,\\
    \vol_\varphi
    &=\tilde\alpha^{123}\tilde\beta\tilde e^{123},
\end{align*}
where $\tilde\alpha^i:=\frac12{\varepsilon}^{ijk}\tilde e_{jk}\lrcorner\varphi,\tilde\beta:=\tilde e_{123}\lrcorner{*\varphi}$ and $\left\{-\tilde\alpha^1,-\tilde\alpha^2,-\tilde\alpha^3, -\tilde\beta, \tilde e^1, \tilde e^2,\tilde e^3\right\}$ is the dual basis of the adapted basis \eqref{adapted}.

Denote $a:=\varphi\left(e_1,e_2,e_3\right)$ and $b:=\left|e_3-g_\varphi\left(e_3,e_1\times e_2\right)e_1\times e_2\right|$ such that $e_3=ae_1\times e_2+b\tilde e_3$. Let $\bar\alpha^i:=\frac12{\varepsilon}^{ijk}e_{jk}\lrcorner\varphi$ and $\bar\beta:=e_{123}\lrcorner{*\varphi}$, then we have a co-basis transform
\begin{align*}
    \tilde e^1
    &=e^1,\tilde e^2=e^2,\tilde e^3=\frac1be^3-\frac ab\bar\alpha^3,\\
    \tilde\alpha^1
    &=e_2\left(\frac1be_3-\frac abe_1\times e_2\right)\lrcorner\varphi=\frac1b\bar\alpha^1-\frac abe^1,\\
    \tilde\alpha^2
    &=\left(\frac1be_3-\frac abe_1\times e_2\right)e_1\lrcorner\varphi=\frac1b\bar\alpha^2-\frac abe^2,\\
    \tilde\alpha^3
    &=e_{12}\lrcorner\varphi=\bar\alpha^3,\\
    \tilde\beta
    &=e_{12}\left(\frac1be_3-\frac abe_1\times e_2\right)\lrcorner*\varphi=\frac1b\bar\beta.
\end{align*}
Under this intermediate co-basis, routine computations shows that
\begin{align}
    \varphi
    &=-\frac1{b^2}\bar\alpha^{123}-\frac1{b^2}\bar\alpha^i\bar\beta e^i+\frac1{2b^2}{\varepsilon}_{ijk}\bar\alpha^i e^{jk}-\frac{2a}{b^2}e^{123}, \nonumber\\
    *\varphi
    &=\frac{b^2-a^2}{b^4}e^{123}\bar\beta
    -\frac{1}{2b^2}\bar\alpha^{ij}e^{ij}+\frac{a}{2b^4}{\varepsilon_{ijk}}e^{ij}\bar\alpha^k\bar\beta
    -\frac{1}{2b^4}{\varepsilon}_{ijk}\bar\alpha^{ij}e^k\bar\beta
    +\frac{a}{b^4}\bar\alpha^{123}\bar\beta, \nonumber\\
    g_\varphi
    &=\frac{1}{b^2}\left(\left(e^1\right)^2+\left(e^2\right)^2+\left(e^3\right)^2+(\bar\alpha^1)^2+(\bar\alpha^2)^2+(\bar\alpha^3)^2+\bar\beta^2\right)-\frac{a}{b^2}\left(e^i\otimes\bar\alpha^i+\bar\alpha^i\otimes e^i\right),\\
    \vol_\varphi&=\frac1{b^4}\bar\alpha^{123}\bar\beta e^{123}.
\end{align}

Let $\left\{U_1,U_2,U_3\right\}$ be a general basis of $\mathbb{U}$. Denote\footnote{In contrast to Madsen-Swann's use of the cyclic permutation notation $(ijk)$ and the corresponding lower index for $\rmd \nu_i$, we adopt $\alpha^i$, $\varepsilon_{ijk}$ and $\varepsilon^{ijk}$, to make it more consistent with Einstein's summation convention.} $\alpha^i:=\frac12{\varepsilon}^{ijk}U_{jk}\lrcorner\varphi$, $\beta:=U_{123}\lrcorner{*\varphi}$ and $c:=\varphi(U_1,U_2,U_3)$. Then, $\left\{\alpha^1,\alpha^2,\alpha^3\right\}$ is a basis of $\Lambda_{\mathbb{U}}$ and $\mathbb{K}_\mathbb{U}=\bigcap_i \ker\alpha^i$. Moreover, let $\theta^i$'s be the covectors on $\mathbb{V}$ defined by 
\begin{align*}
    \theta^i(U_j)
    & =\delta^i_j, \\
    \theta^i|_{\mathbb{U}^{\perp_{g_\varphi}}}
    & = 0.
\end{align*}
Let\footnote{Notice that the notations $A$ and $B$ used in this article are converse to Madsen-Swann's notations \cite[Page 3468-3470]{Madsen-Swann}.} $A=(A_{ij})$ be the positive definite symmetric $3\times 3$-matrix with $A_{ij}=g_\varphi\left(U_i, U_j\right)$, and  $B=\left(B_i^j\right)$ be the positive square root of $A^{-1}$. The vectors $\left\{e_i=B_i^j U_j|i=1,2,3\right\}$ form an orthonormal basis of $\mathbb{U}$. Then we have another co-basis transform
\begin{align*}
    e^i&=\left(B^{-1}\right)^i_j\theta^j,\\
    \bar\alpha^i&=\frac12{\varepsilon}^{ijk}B_j^lB_k^mU_{lm}\lrcorner\varphi=\left(\adj B\right)_j^i\alpha^j,\\
    \bar\beta&=B_1^iB_2^jB_3^kU_{ijk}\lrcorner{*\varphi}=\left(\det B\right)\beta.
\end{align*}
Again by routine computations we can rewrite
\begin{align}
\label{varphi-decomposition-nonassociative}
    \varphi={}
    &-\frac{\det B^2}{b^2}\alpha^{123}+\frac{\det A^{-1}A_{ij}}{b^2}\beta\alpha^i\theta^j+\frac{1}{2b^2}{\varepsilon}_{ijk}\alpha^i\theta^{jk}-\frac{2a\det B^{-1}}{b^2}\theta^{123}, \\
\label{star-varphi-decomposition-isotropic} *\varphi={}&-\frac{b^2-a^2}{b^4}\beta\theta^{123}+\frac{\det B^2}{2b^4}{\varepsilon}_{ijk}\alpha^{ij}\beta\theta^k-\frac1{4b^2}\varepsilon_{ijk}\varepsilon_{lmn}A^{nk}\alpha^{ij}\theta^{lm}\nonumber\\
    &+\frac{a\det B}{2b^4}{\varepsilon}_{ijk}\alpha^i\beta\theta^{jk}+\frac{a\det B^3}{b^4}\alpha^{123}\beta,\\
    g_\varphi={}&\frac1{b^2}\left(\theta^t B^{-2}\theta+\det A^{-1}\alpha^tA\alpha+\det A^{-1}\beta^2\right)-\frac{2a\det B}{b^2}\theta^tA\alpha,\\
    \vol_\varphi=&{}\frac{\det A^{-1}}{b^4}\alpha^{123}\beta\theta^{123}.
\end{align}
Also, $c=a\det B^{-1}$.

Define $\hat\alpha^i:=\alpha^i-c\theta^i$, then for any $i$ and $j$, 
\[
\hat\alpha^i\left(U_j\right)=0,\; \beta(U_j)=0,
\]
i.e. they are horizontal. 
Using the coframe $\left\{\theta^1, \theta^2,\theta^3, \hat\alpha^1,\hat\alpha^2,\hat\alpha^3,\beta\right\}$ of $\mathbb{V}$, we can rewrite
\begin{align}
    \varphi
    &=a\det B^{-1}\theta^{123}+\frac12{\varepsilon}_{ijk}\hat\alpha^i\theta^{jk}-\frac{a\det B}{2b^2}{\varepsilon}_{ijk}\hat\alpha^{ij}\theta^k+\frac{\det A^{-1}A_{ij}}{b^2}\beta\hat\alpha^i\theta^j-\frac{\det A^{-1}}{b^2}\hat\alpha^{123},\\
    *\varphi
    &=-\beta\theta^{123}-\frac{A^{pq}}{4b^2}{\varepsilon_{pkl}}\varepsilon_{qst}\hat\alpha^{kl}\theta^{st}
    +\frac{a\det B}{2b^2}\varepsilon_{ijk}\beta\hat\alpha^i\theta^{jk}+\frac{\det A^{-1}}{2b^2}\varepsilon_{ijk}\beta \hat\alpha^{ij}\theta^k
    +\frac{a\det B^3}{b^4}\hat\alpha^{123}\beta, \nonumber\\
    g_\varphi
    &=A_{ij}\theta^i\theta^j+\frac{\det A^{-1}A_{ij}}{b^2}\hat\alpha^i\hat\alpha^j+\frac{\det A^{-1}}{b^2}\beta^2, \nonumber\\
    \vol_\varphi
    &=\frac{\det A^{-1}}{b^4}\hat\alpha^{123}\beta\theta^{123}.\nonumber
\end{align}
This expression holds as long as $\mathbb{U}$ is non-associative.\\

\subsubsection{Isotropic subspace}~
If $\mathbb{U}$ is assumed to be isotropic, then $c=0$ and $a=0$ and $b^2=1$. Moreover, in this case $\alpha^i=\hat\alpha^i$ for $i=1,2,3$ becomes horizontal. The formulae \eqref{varphi-decomposition-nonassociative} and \eqref{star-varphi-decomposition-isotropic} are precisely the decompositions obtained by \cite[Proposotion 3.2]{Madsen-Swann}. The first term in \eqref{varphi-decomposition-nonassociative} only involves horizontal directions, and the first term in \eqref{star-varphi-decomposition-isotropic} does not involve components in $\alpha^i$'s. From the second and third terms of $\varphi$, and the second and third terms of $*\varphi$, we can single out the triple of $2$-forms 
\begin{align*}
    \omega_i
    & = 
    \frac{1}{2}\varepsilon_{ijk}\theta^{jk}
    + 
    \det A^{-1}A_{ip}\theta^p\beta,\\
    \tilde\omega^k
    & = 
    \frac{1}{2}\det A^{-1}\beta\theta^k
    - 
    \frac{1}{2}\varepsilon_{klm}A^{nk}\theta^{lm},
\end{align*}
which satisfy 
\begin{align}
    \omega_i\wedge\omega_j
    & = 
    2\det A^{-1} A_{ij}\theta^{123}\beta
    = 
    2\det A^{-\frac{7}{6}} \frac{A_{ij}}{\det A^\frac{1}{3}}\vol_\mathbb{U}\beta, \\
    \tilde\omega^i\wedge\tilde\omega^j
    & = 
    2\det A^{-1}A^{ij}\theta^{123}\beta
    = 
    2\det A^{-\frac{11}{6}}\frac{A^{ij}}{\det A^{-\frac{1}{3}}} \vol_\mathbb{U}\beta.
\end{align}
Let $\left\{U_1, U_2, U_3, W_1, W_2, W_3,W\right\}$ be the dual basis of $\left\{\theta^1,\theta^2,\theta^3, \alpha^1,\alpha^2,\alpha^3,\beta\right\}$, then $\underline\omega$ and $\underline{\tilde\omega}$ are definite triple on $\mathbb{U}\oplus \langle W\rangle= \langle W_1, W_2, W_3\rangle^{\perp_{g_\varphi}}=\mathbb{K}_\mathbb{U}$.

\begin{lem}[{Canonical Decomposition adapted to an isotropic subspace}.]
\label{lem:canonical-decomposition-isotropic}
   Let $\varphi$ be a linear $\rmG_2$-structure on the vector space $\mathbb{V}$, and $\mathbb{U}$ is an isotropic $3$-dimensional subspace.  Take any basis $U_1, U_2, U_3$  of $\mathbb{U}$, there holds the following canonical decomposition
\begin{align}
    \varphi
    & = 
    -\det A^{-1}\alpha^{123} + \alpha^i\wedge\omega_i, \\
    *\varphi
    & = 
    -\beta\theta^{123} + \frac{1}{2}\varepsilon_{ijk}\alpha^{ij}\wedge\tilde\omega^k,
\end{align}
where the triple of $2$-forms
\begin{align}
    \omega_i 
    & = \frac{1}{2}\varepsilon_{ijk}\theta^{jk} - \left(\det A^{-1}\right)A_{ip}\beta\theta^p, \\
    \tilde\omega^i
    & = 
    -\frac{1}{2}\varepsilon_{pqr}A^{ip}\theta^{qr}
    + 
    \left(\det A^{-1}\right) \beta\theta^i
\end{align}
satisfy 
\begin{align}
    \omega_i\wedge\omega_j
    & = 
    2\det A^{-1} A_{ij}\theta^{123}\beta
    = 
    2\det A^{-\frac{7}{6}} \frac{A_{ij}}{\det A^\frac{1}{3}}\vol_\mathbb{U}\beta, \\
    \tilde\omega^i\wedge\tilde\omega^j
    & = 
    2\det A^{-1}A^{ij}\theta^{123}\beta
    = 
    2\det A^{-\frac{11}{6}}\frac{A^{ij}}{\det A^{-\frac{1}{3}}} \vol_\mathbb{U}\beta, 
\end{align}
i.e. they are definite on $\mathbb{K}_\mathbb{U}$.
\end{lem}

The relation between $\mathbb{U}$ and $\mathbb{K}_{\mathbb{U}}$ when $\mathbb{U}$ is isotropic is illustrated in the following figure.

\begin{figure}[h!]
\centering
    \begin{tikzpicture}
    \draw[thick] (-1,0) -- (3,0) (0,-1) -- (0, 3);
    \draw (0,3) node[right=0.2em]{$\mathbb{U}\subset\mathbb{K}_\mathbb{U}$};
    \draw[thick, blue] (0.04,-1) -- (0.04,3);
    \draw (3,0) node[below=0.3em]{$\mathbb{U}^{\perp_{g_\varphi}}=\bigcap_i \ker\theta^i$,};
    \draw (3,2.5) node[below=0.3em]{$\underline\omega$ and  $\underline{\tilde\omega}$ are definite on $\mathbb{K}_\mathbb{U}$};
    \end{tikzpicture}
    \caption{The forms $\alpha^i$'s become horizontal, i.e. $\mathbb{U}\subset\mathbb{K}_\mathbb{U}$, in case $\mathbb{U}$ is isotropic.}
\end{figure}

\begin{remark}
\label{rmk:definite-triple-from-isotropic}
    The triple can also be expressed in the following form
    \begin{align*}
        \omega_i
        & = 
        W_i\lrcorner \varphi + \frac{1}{2}\left(\det A^{-1}\right)\varepsilon_{ijk}\alpha^{jk}, \\
        \tilde\omega^k
        & = 
        \frac{1}{2}\varepsilon^{ijk}W_{ij}\lrcorner *\varphi + \beta\theta^k.
    \end{align*}
These formulae shows that $\omega_i|_{\mathbb{K}_\mathbb{U}}=\left(W_i\lrcorner \varphi\right)|_{\mathbb{K}_\mathbb{U}}$.  
\end{remark}~

\subsubsection{Non-isotropic subspace}~
If $\mathbb{U}$ is further assumed to be non-isotropic, then $c\neq 0$ and $\left\{\alpha^1,\alpha^2,\alpha^3,\hat\alpha^1,\hat\alpha^2,\hat\alpha^3,\beta\right\}$ forms another coframe of $\mathbb{V}$ under which we can rewrite
\begin{align}
    \label{data-in-alpha-hat-alpha}
    \varphi&=\frac{1}{c^2}\alpha^{123}-\frac{1}{2b^2c^2}\varepsilon_{ijk}\hat\alpha^{ij}\alpha^k+\frac{\det A^{-1}A_{ij}}{b^2c}\beta\hat\alpha^i\alpha^j+\frac{2}{b^2c^2}\hat\alpha^{123},\\
    *\varphi&=-\frac{1}{c^3}\beta\alpha^{123}+\frac{1}{2b^2c^3}\varepsilon_{ijk}\beta\hat\alpha^i\alpha^{jk}-\frac{A^{kn}}{4b^2c^2}\varepsilon_{ijk}\varepsilon_{lmn}\hat\alpha^{ij}\alpha^{lm}-\frac{1}{2b^2c^3}\varepsilon_{ijk}\beta\hat\alpha^{ij}\alpha^k+\frac{b^2-a^2}{b^4c^3}\beta\hat\alpha^{123},\\
    \label{metric-in-alpha-hat-alpha}g_\varphi&=\frac{A_{ij}}{c^2}\alpha^i\alpha^j-\frac{2A_{ij}}{c^2}\alpha^i\hat\alpha^j+\frac{A_{ij}}{b^2c^2}\hat\alpha^i\hat\alpha^j+\frac{\det A^{-1}}{b^2}\beta^2,\\
    \vol_\varphi&=\frac{\det A^{-1}}{b^4c^3}\hat\alpha^{123}\beta\alpha^{123}.
\end{align}

Observing the expression for $\varphi$, we single out the $2$-forms
\begin{align}
\omega_k:=\frac1{2b^2}\varepsilon_{ijk}\hat\alpha^{ij}+\frac{c\det A^{-1}A_{ik}}{b^2}\hat\alpha^i\beta,\;\; k=1,2,3,
\end{align}
and can write 
\begin{align}
\label{varphi-decomposition-in-nonassociative}    
\varphi
= 
\frac{1}{c^2}\alpha^{123} - \frac{1}{c^2}\alpha^i\wedge\omega_i + 
\frac{2}{b^2c^2}\hat\alpha^{123}.
\end{align}
Then $\underline{\omega}:=(\omega_1,\omega_2,\omega_3)$ is a triple of horizontal $2$-forms satisfying
\begin{align}
    \frac12\omega_i\wedge\omega_j
    & =\frac{c\det A^{-1}A_{ij}}{b^4}\hat\alpha^{123}\beta=\det A^{-1/3}A_{ij}\vol_{\underline{\omega}}
     = c \det A A_{ij}*\theta^{123}.
\end{align}
This means it is a \emph{definite triple} on $\mathbb{U}^{\perp_{g_\varphi}}$ in the sense of \cite{Don, Fine-Yao}, and induces a definite bilinear form 
\begin{align*}
    g_{\underline{\omega}}
    &=\frac{\det A^{-1/3}A_{ij}}{b^2}\hat\alpha^i\hat\alpha^j+\frac{c^2\det A^{-4/3}}{b^2}\beta^2
    \end{align*} 
on $\mathbb{U}^{\perp_{g_\varphi}}$, which is positive with respect to the volume form (i.e. orientation)
    \begin{align*}
    \vol_{\underline{\omega}}
    &=\frac{c\det A^{-2/3}}{b^4}\hat\alpha^{123}\beta 
    =c\det A^\frac{4}{3} *\theta^{123}.
\end{align*}
Actually, the triple can be expressed more intrinsically as
\begin{align}
\label{intrinsic-expression-for-omega}
    \omega_k
    & =-cU_k\lrcorner\varphi+\frac12\varepsilon_{ijk}\alpha^{ij}, \;\\
    \vol_{\underline\omega}
    & = 
    c\det A^\frac{5}{6} \vol_{\mathbb{U}^{\perp_{g_\varphi}}}
\end{align}
without referring to the factor $\frac{1}{b^2}$ which is $\infty$ ( and does not make sense in the situation $\mathbb{U}$ is associative).

Similarly, observing the expression for $*\varphi$, we single out the $2$-forms 
\begin{align}
 \tilde\omega^i
     & =-\frac{1}{b^2}\beta\hat\alpha^i+\frac{cA^{pi}}{2b^2}{\varepsilon_{pkl}}\hat\alpha^{kl}, \;\; i=1,2,3,
\end{align}
and write
\begin{align*}
\label{star-varphi-decomposition-in-nonassociative}    
*\varphi
& = 
    -\frac{1}{c^3}\beta\alpha^{123}-\frac{1}{2c^3}\varepsilon_{ijk}\tilde\omega^i\wedge\alpha^{jk}-\frac{1}{c^3}\omega_i\wedge\beta\alpha^i
    +\frac{b^2-a^2}{b^4c^3}\beta\hat\alpha^{123}.
\end{align*}
The triple $\underline{\tilde\omega}=\left(\tilde\omega^1,\tilde\omega^2,\tilde\omega^3\right)$ is also a triple of horizontal $2$-forms satisfying
\begin{align}
    \frac{1}{2}\tilde\omega^i\wedge\tilde\omega^j
    & =     
      \det A^\frac{1}{3}A^{ij}\vol_{\underline{\tilde\omega}}
      =
      \frac{c A^{ij}}{b^4}\hat\alpha^{123}\beta
\end{align}
meaning that it is also a \emph{definite triple} on $\mathbb{U}^{\perp_{g_\varphi}}$ as $\underline\omega$ introduced above. It induces a positive definite bilinear form
\begin{align*}
    g_{\underline{\tilde\omega}}&=\frac{c\det A^{-2/3}A_{ij}}{b^2}\hat\alpha^i\hat\alpha^j+\frac{c^{-1}\det A^{1/3}}{b^2}\beta^2
\end{align*}
on $\mathbb{U}^{\perp_{g_\varphi}}$ with respect to the volume form
\begin{align*}
    \vol_{\underline{\tilde\omega}}
    &=\frac{c\det A^{-1/3}}{b^4}\hat\alpha^{123}\beta
    =  c\det A^\frac{5}{3}*\theta^{123}.
\end{align*}
Also as above, we have the more intrinsic expressions
\begin{align}
\label{intrinsic-expression-for-tildeomega}
    \tilde\omega^k
    & =-\frac12c\varepsilon^{ijk}U_{ij}\lrcorner{*}\varphi+\alpha^k\beta, \\
    \vol_{\underline{\tilde\omega}}
    & = c\det A^\frac{7}{6}\vol_{\mathbb{U}^{\perp_{g_\varphi}}}
\end{align}
without referring to the factor $\frac{1}{b^2}$.

\subsection{Associative subspace}\,
Suppose $\mathbb{U}$ is an associative $3$-dimensional subspace. Let $\left\{e_1,e_2,e_3\right\}$ be an orthonormal basis of $\mathbb{U}$, then $a:=\varphi(e_1,e_2,e_3)=1$. Without loss of generality, assume $a=1$. The key feature in this ``associative'' case is that $e_3=e_1\times e_2$ and $\left\{e_1,e_2,e_3\right\}$ is closed under $\times$. The set $e_1,e_2,e_3$ can be extended (although highly non-uniquely, actually there exists precisely an $\SU(2)$-worth of such choice) to an adapted basis $\left\{e_1,e_2,e_3,f_0,f_1,f_2,f_3\right\}$ such that
\begin{align*}
    \varphi&=e^{123}-e^i\wedge\left(f^{0i}+\frac12\varepsilon_{ijk}f^{jk}\right),\\
    *\varphi&=f^{0123}-\frac12\varepsilon_{ijk}e^{ij}\wedge\left(f^{0k}+\frac12\varepsilon_{kab}f^{ab}\right),\\
    g_\varphi&=(e^1)^2+(e^2)^2+(e^3)^2+(f^0)^2+(f^1)^2+(f^2)^2+(f^3)^2,\\
    \vol_\varphi&=e^{123}f^{0123}
\end{align*}
under the dual basis $\left\{e^1,e^2,e^3,f^0,f^1,f^2,f^3\right\}$.

Let $\left\{U_1, U_2, U_3\right\}$ be a general basis of $\mathbb{U}$ such that $c=\varphi(U_1,U_2,U_3)>0$, and $A, B$ are the matrice defined as above, and $\theta^i$'s be the $1$-forms defined as above. Let $e_i=B_i^jU_j$ be the formed orthonormal basis of $\mathbb{U}$, then the ``associativity'' tells that $c=\varphi(U_1,U_2,U_3)=a\det B^{-1}=\det B^{-1}$. Since $e^i=(B^{-1})^i_j\theta^j$, we can rewrite
\begin{align}
    \varphi
    &=\det B^{-1}\theta^{123}-\theta^a\wedge(B^{-1})_a^i\left(f^{0i}+\frac12\varepsilon_{ijk}f^{jk}\right),\\
    *\varphi
    &=f^{0123}-\frac{\det B^{-1}}2\varepsilon_{ijl}\theta^{ij}\wedge B_k^l\left(f^{0k}+\frac12\varepsilon_{kab}f^{ab}\right),\\
    g_\varphi
    &=A_{ij}\theta^i\theta^j+(f^0)^2+(f^1)^2+(f^2)^2+(f^3)^2, \nonumber\\
    \vol_\varphi
    &=\det B^{-1}\theta^{123}f^{0123}\nonumber
\end{align}
under the basis $\left\{\theta^1,\theta^2,\theta^3,f^0,f^1,f^2,f^3\right\}$ of $\mathbb{V}^*$ dual to the basis $\left\{U_1,U_2, U_3, f_0,f_1,f_2,f_3\right\}$ of $\mathbb{V}$. According to the above formulae, the natural $1$-forms will be
\begin{align*}
    \alpha^i&:=\frac12\varepsilon^{ijk}U_{jk}\lrcorner\varphi=\det B^{-1}\theta^i,\\
    \beta&:=U_{123}\lrcorner{*}\varphi=0,\\
    \hat\alpha^i&:=\alpha^i-c\theta^i=0.
\end{align*}
Therefore we can rewrite $\varphi$ and $*\varphi$ as
\begin{align}
\label{varphi-decomposition-in-associative}
    \varphi&=\frac{1}{c^2}\alpha^{123}-\frac{1}{c^2}\omega_i\wedge\alpha^i,\\
\label{star-varphi-decomposition-in-associative}    
    *\varphi&=-\frac{1}{2c^3}\varepsilon_{ijk}\tilde\omega^i\wedge\alpha^{jk}+\bar\Omega,
\end{align}
where
\begin{align*}
    \omega_i&
    :=\left(\det B^{-1}\right)\left(B^{-1}\right)_i^a\left(f^{0a}+\frac12\varepsilon_{abc}f^{bc}\right)
    =-cU_i\lrcorner\varphi+\frac12\varepsilon_{ijk}\alpha^{jk},\\
    \tilde\omega^i
    &:=\left(\det B^{-2}\right)B^i_a\left(f^{0a}+\frac12\varepsilon_{abc}f^{bc}\right)
    =-\frac12c\varepsilon^{ijk}U_{jk}\lrcorner{*}\varphi,\\
    \bar\Omega
    &:=f^{0123}=*\varphi+\frac{1}{2c^3}\varepsilon_{ijk}\tilde\omega^i\wedge\alpha^{jk}.
\end{align*}
These formulae show that the $2$-forms $\omega_i$'s and $\tilde\omega^i$'s are independent of the particular choice of the complemented vectors $\left\{f_0,f_1,f_2,f_3\right\}$ (even though they appear a priori so), and agree with the intrinsic expressions \eqref{intrinsic-expression-for-omega} and \eqref{intrinsic-expression-for-tildeomega}. Moreover, the decomposition \eqref{varphi-decomposition-in-nonassociative} and \eqref{varphi-decomposition-in-associative}, \eqref{star-varphi-decomposition-in-nonassociative} and \eqref{star-varphi-decomposition-in-associative} also agree in the cases of $\mathbb{U}$ being non-isotropic non-associative and associative.

The triples of horizontal $2$-forms $\underline\omega$ and $\underline{\tilde\omega}$ satisfy 
\begin{align*}
    \frac12\omega_i\wedge\omega_j
    &=\left(\det A\right)A_{ij}\bar\Omega,\\
    \frac12\tilde\omega^i\wedge\tilde\omega^j
    &=\left(\det A\right)^2 A^{ij}\bar\Omega,
\end{align*}
and are therefore also definite on $\mathbb{U}^{\perp_{g_\varphi}}$. Since $pr:\mathbb{K}_\mathbb{U}\rightarrow \mathbb{U}$ is isomorphism, we can use $pr$ to pull back them to definite triples on $\mathbb{K}_\mathbb{U}$.~

\begin{lem}[Canonical decomposition adapted to a non-isotropic subspace.] 
\label{lem:canonical-non-isotropic}
Let $\varphi$ be a linear $\rmG_2$-structure on the vector space $\mathbb{V}$, and $\mathbb{U}$ is a non-isotropic $3$-dimensional subspace. Take any basis $U_1,U_2,U_3$ of $\mathbb{U}$, and denote $c=\varphi\left(U_1,U_2,U_3\right)\neq 0$. Define 
\begin{align*}
    & A_{ij}
      = g_\varphi\left(U_i, U_j\right), \\
    &\alpha^i 
     = \frac{1}{2}\varepsilon^{ijk}U_{jk}\lrcorner \varphi, \;\;
    \beta
     = U_{123}\lrcorner *\varphi, \\
    & \theta^i\left(U_j\right)
     = \delta^i_j, \;\; \theta^i|_{\mathbb{U}^{\perp_{g_\varphi}}}=0, \\
    &\omega_k
     =-cU_k\lrcorner\varphi+\frac12\varepsilon_{ijk}\alpha^{ij}, \;\\
    &\tilde\omega^k
    =-\frac12c\varepsilon^{ijk}U_{ij}\lrcorner{*}\varphi+\alpha^k\beta.
\end{align*}
Then there are the following decomposition
\begin{align}
\label{phi-canonical-decomposition}
    \varphi
    &=
    \frac{1}{c^2}\alpha^{123}-\frac{1}{c^2}\alpha^i\wedge\omega_i+\Omega,\\
    \label{starphi-canonical-decomposition}
    *\varphi
    &=
    -\frac{1}{c^3}\beta\alpha^{123}-\frac{1}{2c^3}\varepsilon_{ijk}\tilde\omega^i\wedge\alpha^{jk}-\frac{1}{c^3}\omega_i\wedge\beta\alpha^i+\bar\Omega.
\end{align}
Moreover, 

\begin{enumerate} 
\item The forms $\beta, \omega_k, \tilde\omega^k, \Omega, \bar\Omega$ are horizontal;
\item The decomposition is canonical, in the sense that each term in the decomposition \eqref{phi-canonical-decomposition}, \eqref{starphi-canonical-decomposition} is independent of the particular choice of basis $U_1,U_2,U_3$ of $\mathbb{U}$; 
\item The triples $\underline\omega=\left(\omega_1, \omega_2,\omega_3\right)$ and $\underline{\tilde\omega}=\left(\tilde\omega^1,\tilde\omega^2,\tilde\omega^3\right)$ are \emph{definite} on $\mathbb{U}^{\perp_{g_\varphi}}$ in the sense that
\begin{align*}
    \omega_i\wedge \omega_j
    & = 
    2c\det A^\frac{5}{6}\frac{A_{ij}}{\det A^\frac{1}{3}}\vol_{\mathbb{U}^{\perp_{g_\varphi}}}, \\
    \tilde\omega^i\wedge \tilde\omega^j
    & = 
    2c\det A^\frac{7}{6}\frac{A^{ij}}{\det A^{-\frac{1}{3}}}\vol_{\mathbb{U}^{\perp_{g_\varphi}}}
\end{align*}
where $\vol_{\mathbb{U}^{\perp_{g_\varphi}}}=\det A^\frac{1}{2}*\theta^{123}$ is a horizontal $4$-form whose restriction to $\mathbb{U}^{\perp_{g_\varphi}}$ is the volume form of $\mathbb{U}^{\perp_{g_\varphi}}$;
\item The triples $\underline\omega$ and $\underline{\tilde\omega}$ are also definite on $\mathbb{K}_\mathbb{U}$. 
\end{enumerate}
\end{lem}

The $1$-forms $\alpha^1,\alpha^2,\alpha^3$ are not horizontal, and we can obtain horizontal $1$-forms $\hat\alpha^i : = \alpha^i - c\theta^i$ (for $i=1,2,3$) by subtracting off their vertical components. These new forms $\hat\alpha^1,\hat\alpha^2,\hat\alpha^3$ together with $\beta$ behaves drastically different in the non-associative and associative subcases:

\begin{lem}
\label{lem:specific-form-of-hypersymplectic-in-nonisotropic}
   If $\mathbb{U}$ is non-isotropic non-associative, then $\left\{\hat\alpha^1, \hat\alpha^2, \hat\alpha^3,\beta\right\}$ are linearly independent, and under which there holds the extra decomposition 
\begin{align}
\label{hypersymplectic-in-nonisotropic-nonassociative}
\omega_k
&=\frac1{2b^2}\varepsilon_{ijk}\hat\alpha^{ij}+\frac{c\det A^{-1}A_{ki}}{b^2}\hat\alpha^i\beta,    \\
\label{dual-hypersymplectic-in-nonisotropic-nonassociative}
\tilde\omega^i
     & =-\frac{1}{b^2}\beta\hat\alpha^i+\frac{cA^{pi}}{2b^2}{\varepsilon_{pkl}}\hat\alpha^{kl}, \\
   \label{Omega-in-nonisotropic-nonassociative}
     \Omega
     & = \frac{2}{b^2c^2}\hat\alpha^{123}, \\
     \bar\Omega
     & = \frac{b^2-a^2}{b^4c^3}\beta\hat\alpha^{123},
   \end{align}
   where $a=c\det A^{-1/2},\;b^2=1-a^2$.
\end{lem}

\begin{lem}
    If $\mathbb{U}$ is associative, then $\hat\alpha^1=\hat\alpha^2=\hat\alpha^3=\beta=0$, and moreover, 
    \begin{align}
        \Omega 
        & = 0, \\
        \bar\Omega
        & = \vol_{\mathbb{U}^{\perp_{g_\varphi}}}.
    \end{align}
\end{lem}

\begin{figure}[h!]
\centering
    \begin{tikzpicture}
    \draw[thick] (-1,0) -- (3,0) (0,-1) -- (0,3);
    \draw (0,3) node[right=0.2em]{$\mathbb{U}=\langle U_1, U_2, U_3\rangle$};
    \draw[thick, blue] (-1,-0.5) -- (2,1);
    \draw (1.5,1.2) node[right=0.2em]{$\mathbb{K}_\mathbb{U}=\left\langle W_1-\frac{1}{c}U_1,W_2-\frac{1}{c}U_2,W_3-\frac{1}{c}U_3, W\right\rangle$};
    \draw (5,0) node[below=0.3em]{$\mathbb{U}^{\perp_{g_\varphi}}=\bigcap_i \ker\theta^i=\langle W_1,W_2,W_3,W\rangle$,};
    \draw (3,-0.5) node[below=0.3em]{$\underline\omega$ and $\underline{\tilde\omega}$ are definite on it};
    \draw[->] (1.5,0.1) -- (1.5, 0.6);
    \draw[->] (2,0.8) -- (2, 0.1);
    \draw (1.5,0.3) node[left=0.2em]{$\iota$} (2,0.4) node[right=0.2em]{$pr$};
    \end{tikzpicture}
    \caption{Direct sum decomposition of $\mathbb{V}=\mathbb{U}\oplus \mathbb{K}_\mathbb{U}$ and orthogonal decomposition $\mathbb{V}=\mathbb{U}\oplus \left(\bigcap_i\ker\theta^i\right)$, in case $\mathbb{U}$ is non-isotropic non-associative; Here $U_1, U_2, U_3, W_1,W_2,W_3,W$ is the dual basis of $\theta^1,\theta^2,\theta^3, \hat\alpha^1,\hat\alpha^2,\hat\alpha^3,\beta$.}
\end{figure}

\begin{remark}
  If $\mathbb{U}$ is a non-isotropic subspace, there are four Riemannian metrics on $\mathbb{U}^{\perp_{g_\varphi}}$. They are respectively induced by the submersion ($pr:\mathbb{V}\to \mathbb{U}^{\perp_{g_\varphi}}$) and immersion ($\iota:\mathbb{U}^{\perp_{g_\varphi}}\hookrightarrow \mathbb{K}_\mathbb{U}$), the definite triples $\underline{\omega}$ and $\underline{\tilde\omega}$. The conformal classes of each one of the four Riemannian metrics are independent of the choice of basis $U_1, U_2, U_3$ of $\mathbb{U}$. 

\begin{itemize}
 \item  If $\mathbb{U}$ is moreover non-assciative, then
    \begin{align*}
        g_{sub}
        &=\frac{\det A^{-1}A_{ij}}{b^2}\hat\alpha^i\hat\alpha^j+\frac{\det A^{-1}}{b^2}\beta^2
         = \frac{\det A^{-1}}{b^2}\left( A_{ij}\hat\alpha^i\hat\alpha^j+\beta^2\right),\\
        g_{imm}
        &=\frac{A_{ij}}{b^2c^2}\hat\alpha^i\hat\alpha^j+\frac{\det A^{-1}}{b^2}\beta^2
         = \frac{\det A^{-1}}{b^2}\left(\frac{1}{a^2} A_{ij}\hat\alpha^i\hat\alpha^j+\beta^2\right),\\
        g_{\underline{\omega}}
        &=\frac{\det A^{-1/3}A_{ij}}{b^2}\hat\alpha^i\hat\alpha^j+\frac{c^2\det A^{-4/3}}{b^2}\beta^2
         = \frac{\det A^{-\frac{1}{3}}}{b^2}\left(A_{ij}\hat\alpha^i\hat\alpha^j+a^2\beta^2\right),\\
        g_{\underline{\tilde\omega}}
        &=\frac{c\det A^{-2/3}A_{ij}}{b^2}\hat\alpha^i\hat\alpha^j+\frac{c^{-1}\det A^{1/3}}{b^2}\beta^2
        = \frac{\det A^{-\frac{1}{6}}}{b^2}\left(a A_{ij}\hat\alpha^i\hat\alpha^j + \frac{1}{a}\beta^2\right).
    \end{align*}
 While the metric $g_{sub}$ and $g_{imm}$ are independent of the particular choice of basis $U_1,U_2,U_3$ of $\mathbb{U}$, the metric $g_{\underline\omega}$ and $g_{\underline{\tilde\omega}}$ are depending the particular choice.

\item  If $\mathbb{U}$ is associative, then $g_{sub}=g_{imm}$, and the four conformal classes coincide.

\begin{figure}[h!]
\centering
    \begin{tikzpicture}
    \draw[thick] (-1,0) -- (3,0) (0,-1) -- (0,3);
    \draw (0,3) node[right=0.2em]{$\mathbb{U}$};
    \draw[thick, blue] (-1,0.04) -- (3,0.04);
    \draw (3,0) node[below=0.3em]{$\mathbb{U}^{\perp_{g_\varphi}}=\bigcap_i \ker\theta^i=\bigcap_{i}\ker\alpha^i$,};
    \draw (3,-0.5) node[below=0.3em]{$\underline\omega$ and  $\underline{\tilde\omega}$ are definite on it};
    \end{tikzpicture}
    \caption{Direct sum decomposition of $\mathbb{V}=\mathbb{U}\oplus \mathbb{K}_\mathbb{U}$ coincides with the orthogonal decomposition $\mathbb{V}=\mathbb{U}\oplus \left(\bigcap_i \ker\theta^i\right)$, in case $\mathbb{U}$ is associative.}
\end{figure}
 \end{itemize}
\end{remark}

We finish this section with a lemma relating the two definite triples obtained from $\varphi$ and $*\varphi$ in the case $\mathbb{U}$ is non-isotropic. 
\begin{lemma}
\label{lem:relation-between-two-hypersymplectic}
The two definite triples
$\underline{\omega}$ and $\underline{\tilde\omega}$ in Lemma \ref{lem:canonical-non-isotropic} satisfy
\begin{align*}
    \tilde\omega^i-cA^{ij}\omega_j+\beta\hat\alpha^i
    &=0,\\
    c\det A^{-1}A_{ij}\tilde\omega^j-\omega_i+\frac{1}{2}\varepsilon_{ijk}\hat\alpha^{jk}
    &=0.
\end{align*}

\end{lemma}

\section{Non-isotropic orbit}\label{sect:non-isotropic}


Let $(M, \varphi)$ be a $7$-manifold equipped with a closed $\rmG_2$-structure such that $\varphi$ is invariant under an effective $\T^3$-action. Let $M_0\subset M$ be the complement of singular orbits, i.e. the subset of points lying on $3$-dimensional orbits, and $M_0'\subset M_0$ be the subset of principal orbits. Let $\left\{U_1, U_2, U_3\right\}$ be a set of generating vector fields of the action, $c=\varphi\left(U_1,U_2,U_3\right)$ and $\mathbb{U}=\bigcup_{p\in M_0} \mathbb{U}_p$ where $\mathbb{U}_p=\Span\left\{U_1|_p, U_2|_p, U_3|_p\right\}$. Define $e_i:=B_i^jU_j$, then $g_\varphi(e_i,e_j)=\delta_{ij}$. Define 
\[
a:=\varphi(e_1,e_2,e_3)=c\det B,
\]
which is a smooth $\mathbb{T}^3$-invariant function on $M$ which takes values in $(0,1]$. Let $\theta^i$'s be the $1$-forms on $M_0$ such that
\[
\theta^i(U_j)=\delta^i_j,\;\;
\text{ and }\;\;
\theta^i(V)=0,\;\; \forall V\perp \mathbb{U}.
\]
These $1$-forms are $\T^3$-invariant, so they are the $1$-forms for the connection on the principal $\mathbb{T}^3$-bundle $\pi: M_0'\rightarrow X'$ determined by the $\T^3$-invariant distribution $\mathbb{U}^{\perp_{g_\varphi}}$. Then $\rmd\theta^i\in\Gamma\left(\Lambda^2T^*X'\right)$ ($i=1,2,3$) represent the curvature $2$-forms. Let $e^i:=e_i^\flat=(B^{-1})^i_j\theta^j$, then $g_\varphi(e^i,e^j)=\delta^{ij}$. We also denote \begin{align*} 
\alpha^i=\frac{1}{2}\varepsilon^{ijk}U_{jk}\lrcorner \varphi, \;\beta=U_{123}\lrcorner *\varphi
\end{align*}
to be the natural $1$-forms. The closedness and $\T^3$-invariance of $\varphi$ implies $c$ is a constant function on $M$, and $\alpha^1,\alpha^2,\alpha^3$ are closed $1$-forms on $M$. Define a distribution 
\begin{align}
\mathbb{K}=\bigcap_{i=1}^3 \ker \alpha^i
\end{align}
on $M_0$, it is integrable as $\alpha^i$'s are closed. 

The existence of one isotropic $\T^3$-orbit implies the constant $c=0$, and therefore all orbits are isotropic. In this section, we consider the case for which no isotropic orbit is present. In particular, $c\neq 0$ and any generating vector fields $U_1,U_2,U_3$ are nowhere linearly dependent and therefore there exists no singular orbit. Without loss of generality, assume $c>0$. Decompose $M=M_0=M'\cup M''$ where $M'=\{p\in M:a(p)\in(0,1)\}$ is the non-associative part which is a $\mathbb{T}^3$-invariant open subset, and $M''=\{p\in M:a(p)=1\}$ is the associative part which is a $\T^3$-invariant closed subset. We express $\varphi$ in these two cases using Lemma \ref{lem:canonical-non-isotropic}.\\

\subsection{Closed and Torsion-free conditions}~
\subsubsection{Closed condition}~
Lemma \ref{lem:canonical-non-isotropic} gives the canonical decomposition
\[
 \varphi
    =
    \frac{1}{c^2}\alpha^{123}-\frac{1}{c^2}\alpha^i\wedge\omega_i+\Omega,
\]
where $\omega_k =-cU_k\lrcorner\varphi+\frac12\varepsilon_{ijk}\alpha^{ij}$ is a horizontal $2$-form for $k=1,2,3$. Since $c$ is a constant function and $\alpha^i$'s are closed, the closedness and $\T^3$-invariance of $\varphi$ implies $\rmd\omega_k=0$ and therefore $\rmd\Omega=0$ on $M$. Together with the ``definiteness'' from Lemma \ref{lem:canonical-non-isotropic}, we derive that $\underline{\omega}=(\omega_1,\omega_2,\omega_3)$ descends to a hypersymplectic structure on the open manifold $M'/\T^3\subset M/\T^3$.\\

The key feature of non-associativity of a $3$-dimensional subspace, that the cross product of vectors generate all vectors, implies that the cross-products of $U_1, U_2, U_3$ generates $T_pM$ if $p$ lies on a non-singular orbit (cf. \cite[Lemma 2.6]{Madsen-Swann}). For such $p$, the isotropy group $H_p$ fixes any vector in $T_pM$ as it fixes $U_1,U_2, U_3$. As a consequence, $H_p=\left\{0\right\}$ and the orbit of $p$ is not exceptional, there $M'\subset M_0$.~

An associative $\T^3$-orbit might be exceptional, and do show up in some concrete examples. Take $p\in M$ such that the stabilizer group $H_p\neq \{0\}$, and let $\mathcal{O}$ be the orbit of $p$ and $T_p\mathcal{O}:=\mathbb{U}_p$. For any $h\in H_p\backslash\{0\}$, 
    \[
    \rmd h|_p: T_pM=\mathbb{U}_p\oplus \mathbb{U}_p^{\perp_{g_\varphi}} \rightarrow T_p M=\mathbb{U}_p\oplus \mathbb{U}_p^{\perp_{g_\varphi}}
    \]
    preserves each vector in $\mathbb{U}_p$, and therefore maps its orthogonal complement $\mathbb{U}_p^{\perp_{g_\varphi}}$ to $\mathbb{U}_p^{\perp_{g_\varphi}}$. As $h$ is a nontrivial isometry of $\left(M, g_\varphi\right)$, $\rmd h|_p\neq id$, which means $\rmd h|_p:\mathbb{U}_p^{\perp_{g_\varphi}}\rightarrow \mathbb{U}_p^{\perp_{g_\varphi}}$ is not identity. Suppose there is a sequence of exceptional orbits $\mathcal{O}_k$ converging to $\mathcal{O}$, then the \emph{Slice Theorem} for group action tells that the isotropy group of $\mathcal{O}_k$ is a subgroup of the finite group $H_p$ for $k$ sufficiently large. We can therefore take a subsequence $k_l$ such that the isotropy groups $\mathcal{O}_{k_l}$ remains constant and nontrivial.  Any nontrivial element $h$ in it fixes any points in $\mathcal{O}_{k_l}$ and this leads to the conclusion that $h$ will fix the geodesics connecting $p$ to these orbits and their initial directions, and therefore $\rmd h|_p$ fixes some nonzero vector in $\mathbb{U}_p^{\perp_{g_\varphi}}$. 
    As $\rmd h|_p$ preserves $\varphi|_p$, it preserves the definite triple $\left(\omega_1,\omega_2,\omega_3\right)|_p$ on $\mathbb{U}_p^{\perp_{g_\varphi}}$ canonically determined by $\varphi|_p$ (see Lemma \ref{lem:canonical-non-isotropic}
), i.e. $\rmd h|_p\in \SU\left(\mathbb{U}_p\right)$ (see also \cite[Lemma 3.6]{FHY}). This implies $\rmd h|_p=id$, a contradiction. As a consequence, the exceptional orbit must be isolated. Moreover, $H_p$ can be viewed as a subgroup of $\SU(2)$.\\ 

Since in the meantime $H_p$ is also a discrete subgroup of the Abelian group $\mathbb{T}^3$, $H_p$ must be cyclic, i.e. $A_n$-type subgroup of $\SU(2)$. Notice that not all discrete subgroup of $\T^3$ are cyclic, e.g. $\Z_{d_1}\oplus\Z_{d_2}\oplus\Z_{d_3}$. However, here $H_p$ is cyclic since $\langle\rmd h|_p\rangle$ is cyclic, from the fact that any finite Abelian group of $\SU(2)=\Sp(1)=S^3$ is cyclic. This implies $X:= M/\mathbb{T}^3$ is a priori an orbifold. Denote the quotient by $\pi:M\to X$.

Fix a point $p\in M$, take the integral leaf
\[
    \tilde X_p:=\left\{q\in M|\exists\gamma:[0,1]\to M\textrm{ smooth curve s.t. }\alpha_i(\dot\gamma)=0,\gamma(0)=p,\gamma(1)=q\right\}
\]
through $p$. For any $q\in \tilde X_p$ such that $H_q\neq \left\{0\right\}$, $\mathbb{K}_q=\mathbb{U}_q^{\perp_{g_\varphi}}$ is fixed by all element of $H_q$. As a consequence, $H_q$ fixes the whole leaf $\tilde X_p$, i.e. $H_q$ acts on $\tilde X_p$. Moreover, each orbit being non-isotropic implies $\mathbb{K}_q$ is transversal to $\mathbb{U}_q$ at any $q\in M$. Notice that the $\T^3$-invariance of the distribution $\mathbb{K}$ implies $\tilde X_{h\cdot p}=h\cdot \tilde X_p$, and therefore $M$ is foliated by a family of mutually diffeomorphic leaves. The underlying unique differentiable manifold could be simply denoted by $\boldsymbol X$, whose topological structure might be different from the induced topological structure from $M$ as $\tilde X_p$ might not be an closed submanifold of $M$.\\
 
 Take any one of these leaves, say $\tilde X_p$, denote the immersion $i_p:\boldsymbol X\rightarrow M$ with image $\tilde X_p$, and define 
    \[
    G:=\text{Stab}\left(\tilde X\right):=\left\{\underline t\in \mathbb{T}^3| \underline t\cdot q\in \tilde X_p \text{ for some }q\in \tilde X_p\right\}.
    \]
    Since $\tilde X_p\bigcap \pi^{-1}(x)$ consists of local integral manifold of the horizontal distributions, $G$ is a discrete subgroup of $\mathbb{T}^3$. Any $H_p$ is a subgroup of $G$. We claim $G$ acts on $\boldsymbol X$ and $X=\boldsymbol X/G$.\\

    \begin{proof}[Proof of the claim.]
     We take $\tilde X_p$ to represent $\boldsymbol X$. For any $\tilde x\in \tilde X_p$, $\tilde x$ can be connected to $p$ by a smooth horizontal curve $\gamma$. For any $\underline t\in G$, $\underline t\cdot q\in \tilde X_p$ for some $q\in \tilde X_p$ that can be connected to $p$ by some horizontal curve. It thus holds that $q$ can be connected to $\tilde x$ horizontally, and $\underline t\cdot \tilde x$ can be connected to $\underline t\cdot q$ and therefore to $p$ horizontally.

     Suppose $\tilde x, \tilde y\in \tilde X_p$ such that $\pi(\tilde x)=\pi(\tilde y)$, then $\tilde y=\underline t\cdot \tilde x$ for some $\underline t\in \T^3$. It follows from definition that $\underline t\in G$. Therefore, the preimage of any point under the map $\pi:\tilde X_p\rightarrow X$ is precisely a $G$-orbit. The transversality of $\mathbb{K}$ to $\mathbb{U}$ implies $\pi|_{\tilde X_p}$ is an open map. Suppose $x\in X\backslash\pi\left(\tilde X_p\right)\neq \emptyset$, and $\mathcal{U}$
    is a sufficiently small open $\T^3$-invariant tubular neighborhood of $\mathcal{O}_x$ such that it is foliated by disjoint union of $\left\{\underline t\cdot S|\underline t\in \T^3\right\}$ where $S$ is a local leaf transversal to $\mathcal{O}_x$. As long as $\tilde X_p$ intersects $\mathcal{U}$, it can be extended across a point on $\mathcal{O}_x$ by gluing a local leaf which is a contradiction. Thus, $\pi\left(\mathcal{U}\right)\subset X\backslash \pi\left(\tilde X_p\right)$, i.e. $\pi\left(\tilde X\right)$ is closed. In conclusion, $\tilde X_p\backslash G=X$. 
\end{proof}

The Slice Theorem then implies $\varpi:=\pi\circ i:\boldsymbol X\to X$ is an orbifold-covering. This can be seen as follows. On the principal part, $\pi|_{M'}:M'\to M'/\T^3=:X'$ is a $\T^3$-bundle map. Let $\boldsymbol X':=\varpi^{-1}\left(X'\right)$. Take $x\in X'$, then the submanifold $\pi^{-1}\left(x\right)\sub M$ is a principal orbit, and near $\pi^{-1}\left(x\right)$ the tubular neighborhood $\pi^{-1}\left(\mathfrak{W}\right)$ for some open neighborhood $\mathfrak{W}\sub X'$ is a trivial bundle since $\rmd\theta^i=0$. Note that $i\left(\boldsymbol X'\right)$ is locally the leaf of the distribution $\bigcap_i\ker\theta^i=\bigcap_i \ker\alpha^i=\mathbb{K}$, we have $\varpi^{-1}(\mathfrak{W})$ is a disjoint union of open sets diffeomorphic to $\mathfrak{W}\sub X'$. So $\varpi|_{\boldsymbol X'}:\boldsymbol X'\to X'$ is a covering map.\\

For $x\in X\backslash X'$, $\pi^{-1}\left(x\right)$ is an exceptional orbit. There is a tubular neighborhood $\pi^{-1}\left(\mathfrak{W}\right)$, for some $\mathfrak{W}\sub X$, equivariantly diffeomorphic to $\T^3\times_{H_p}\exp_p(O_p)$ where $p\in \pi^{-1}\left(x\right)$ and $O_p$ is an open neighborhood of $0\in \mathbb{U}_p^{\perp_{g_\varphi}}$. Similarly to the last paragraph, $\varpi^{-1}(\mathfrak{W})$ is a disjoint union of open sets diffeomorphic to $\exp_p(O_p)$. However, $\mathfrak{W}=\pi^{-1}\left(\mathfrak{W}\right)/\T^3=\left(\T^3\times_{H_p}\exp_p\left(O_p\right)\right)/\T^3=\exp_p\left(O_p\right)/H_p$. Therefore $\varpi:\boldsymbol X\to X$ is an orbifold-covering, and $X$ is a \emph{good orbifold}.\\

We claim that the natural free $G$-action on $\T^3\times \boldsymbol X$ via the formula $\underline g\cdot \left(\underline t, \mathbf{x}\right)=\left( \underline t\cdot \underline{g}^{-1}, \underline g\cdot \mathbf{x}\right)$ induces a smooth covering map $\boldsymbol\varpi:\T^3\times\boldsymbol X\to M$ sending $(h,\mathbf{x})$ to $h\cdot i(\mathbf{x})$. 

\begin{proof}[Proof of claim]
We can take $\boldsymbol X=\tilde X_p$. Assume $h_1\cdot \tilde x_1=h_2\cdot \tilde x_2$ for $\tilde x_1, \tilde x_2\in \tilde X_p$, then $\tilde x_2=\underline g\cdot \tilde x_1$ for some $\underline g\in G$, and therefore $h_1^{-1}h_2\underline g=h\in H_{\tilde x_1}$. As a consequence, 
\[
\left(h_2, \tilde x_2\right)
= 
\left(h_1h\underline g^{-1}, \underline g\cdot \tilde x_1\right) 
= 
\underline gh^{-1}\cdot \left(h_1, \tilde x_1\right). 
\]
This shows that the preimage of any point $q\in M$ is precisely a $G$-orbit. The surjectivity of $\pi: \tilde X_p\to X$ proved above implies the surjectivity of $\boldsymbol\varpi$. 
\end{proof}

The map $\boldsymbol \varpi$ is $\T^3$-equivariant, and  
\[
\pi\circ \boldsymbol\varpi=\varpi\circ \mathrm{pr}
\]
where $\mathrm{pr}: \mathbb{T}^3\times \boldsymbol X\to \boldsymbol X$ is the natural projection map. We can then use the map $\boldsymbol \varpi$ to pull back the closed $\rmG_2$-structure $\varphi$ to a closed $\rmG_2$-structure $\boldsymbol \varphi$ on $\mathbb{T}^3\times \boldsymbol X$. The canonically determined triple of $2$-forms $\underline\omega$ on $M$ is $G$-invariant and horizontal. Suppose $\T^3=S^1\times S^1\times S^1=\left\{\left(e^{it_1}, e^{it_2}, e^{it_3}\right)|t_1, t_2, t_3\in \mathbb{R}\right\}$, and using $U_i=\boldsymbol\varpi_*\frac{\partial}{\partial t^i}$ as the generating vector field for the $\T^3$-action on $M$. Let $\boldsymbol\alpha^i=\boldsymbol\varpi^*\alpha^i, \boldsymbol\omega_i=\boldsymbol\varpi^*\omega_i, \boldsymbol\Omega=\boldsymbol\varpi^*\Omega$, then
\begin{align*}
    \boldsymbol \alpha^i\left(\frac{\partial}{\partial t^j}\right)
    & = 
    \alpha^i\left(U_j\right) = c\delta^i_j, \\
    \boldsymbol \alpha^i\left(\boldsymbol W\right)
    & = 
    \alpha^i\left(\boldsymbol \varpi_*\boldsymbol W\right)=0, \forall \boldsymbol W\in T\left(\left\{\underline t\right\}\times \boldsymbol X\right), \\
    \frac{\partial}{\partial t^j}\lrcorner\boldsymbol\omega_i
    & =
    \boldsymbol\varpi^*\left( U_j\lrcorner\omega_i\right)=0, \\
    \frac{\partial}{\partial t^j}\lrcorner \boldsymbol \Omega
    &= 
    \boldsymbol\varpi^*\left(U_j\lrcorner \Omega\right)=0.
\end{align*}
This implies $\boldsymbol\alpha^i=c\rmd t^i$, $\boldsymbol\omega_i$ and $\boldsymbol\Omega$ are horizontal and
\begin{align*}
    \boldsymbol\varphi
    & = 
    c\rmd t^{123} - \frac{1}{c}\rmd t^i\wedge\boldsymbol\omega_i + \boldsymbol\Omega.
\end{align*}

By construction, $\boldsymbol{\underline\omega}$ and $\boldsymbol\Omega$ are $G$-invariant and descend (under the map $\varpi$) to an orbifold hypersymplectic structure, still denoted by $\underline\omega$, and an orbifold closed 3-form, still denoted by $\Omega$ on $X$.

\begin{thm}[Closed $\rmG_2$-structure with a non-isotropic $\T^3$-orbit]
\label{thm:nonisotropic-closed-structure}
    Let $\varphi$ be a closed $\rmG_2$-structure on the 7-manifold $M$. Suppose $\varphi$ is invariant under an effective $\mathbb{T}^3$-action, and there is a non-isotropic $\mathbb{T}^3$-orbit. Let $G=\text{Stab}\left(\tilde X\right)$ for any leaf $\tilde X\subset M$ of the integrable distribution $\mathbb{K}$. Then, 
    \begin{enumerate} 
    \item the action is almost-free, and $X:=M/\T^3$ is a good 4-orbifold with cyclic isolated singularities;
    \item there is a 4-manifold $\boldsymbol X$ which $G$ acts on smoothly, an orbifold-covering map $\varpi:\boldsymbol X\to X\simeq \boldsymbol X/G$, a $\T^3$-equivariant covering map $\boldsymbol\varpi:\T^3\times\boldsymbol X\to M\simeq \left(\T^3\times \boldsymbol X\right)/G$, such that the following diagram commutes
\[\begin{tikzcd}
	{\T^3\times\boldsymbol X} & M \\
	{\boldsymbol X} & X.
	\arrow["{\boldsymbol\varpi}", from=1-1, to=1-2]
	\arrow["{\mathrm{pr}}", from=1-1, to=2-1]
	\arrow["\pi", from=1-2, to=2-2]
	\arrow["\varpi", from=2-1, to=2-2]
\end{tikzcd}\]
\item there exists a hypersymplectic structure $\underline{\boldsymbol\omega}=\left(\boldsymbol\omega_1,\boldsymbol\omega_2,\boldsymbol\omega_3\right)$, a closed 3-form $\boldsymbol\Omega$ on $\boldsymbol X$, a constant $c\neq 0$, such that
\begin{align*} 
\boldsymbol\varpi^*\varphi
=\boldsymbol\varphi
=
c\rmd t^{123}-\frac{1}{c}\rmd t^i\wedge \boldsymbol\omega_i +\boldsymbol\Omega.
\end{align*}

\item  the structures $\underline{\boldsymbol\omega}, \boldsymbol\Omega$ are $G$-invariant, and thus descend to hypersymplectic structure $\underline\omega$ and closed 3-form $\Omega$ on the orbifold $X$. Moreover, the $\rmG_2$-structure can be written as
    \[
        \varphi=\frac{1}{c^2}\alpha^{123}-\frac{1}{c^2}\pi^*\omega_i\wedge\alpha^i+\pi^*\Omega,
    \]
    where $U_i=\boldsymbol\varpi_*\frac{\partial}{\partial t^i}$ and $\alpha^i:=\dfrac12\varepsilon^{ijk}U_{jk}\lrcorner\varphi$.
\end{enumerate}
\end{thm}




The next example shows that associative and non-associative orbits can co-exist for a $\T^3$-invariant closed $\rmG_2$-structure, though we do not have a concrete example for torsion-free case. 
\begin{example}[Co-existence of associative and non-associative orbits]
\label{ex:coexistence}
    Let $\T^3\times X$ be a $7$-dimensional manifold with a canonical $\T^3$-action preserving a closed $\rmG_2$-structure $\tilde\varphi$, such that the orbit is isotropic. Let $U_1,U_2,U_3$ be the generating field. Let $\tilde\theta^1,\tilde\theta^2,\tilde\theta^3$ be the canonical flat connections, which is closed $1$-form. Let $\tilde\alpha^i:=\frac12\varepsilon^{ijk}U_{jk}\lrcorner\tilde\varphi$ which are closed. Let $\tilde\beta:=U_{123}\lrcorner{*}\tilde\varphi$. Then $\tilde\theta^1,\tilde\theta^2,\tilde\theta^3,\tilde\alpha^1,\tilde\alpha^2,\tilde\alpha^3,\tilde\beta$ is a global coframe, and
    \[
        \tilde\varphi=-\det\tilde A^{-1}\tilde\alpha^{123}+\det\tilde A^{-1}\tilde A_{ij}\tilde\beta\tilde\alpha^i\tilde\theta^j+\frac{1}{2}\varepsilon_{ijk}\tilde\alpha^i\tilde\theta^{jk},
    \]
    where $\tilde A=(\tilde A_{ij})$ is a positive-definite symmetric matrix-valued function on $X$ such that
    \begin{align*}
        0&=\rmd(\det\tilde A^{-1})\tilde\alpha^{123},\\
        0&=\rmd(\det\tilde A^{-1}\tilde A_{ij}\tilde\beta)\tilde\alpha^i, \;\;\forall j.
    \end{align*}
    Such $\tilde A$ exists. For example, $X=\T^4$ with parameter $x^1,x^2,x^3,y$ and
    \[
        \tilde A:=
        \begin{bmatrix}
            f_1(x^1,y)\\
            &f_2(x^2,y)\\
            &&f_3(x^3,y)
        \end{bmatrix}
    \]
    where $f_i$ are positive non-trivial with a positive lower bound such that $\partial_y(f_1f_2f_3)=0$, and $\tilde\alpha^i:=\rmd x^i$ and $\tilde\beta:=\det\tilde A\rmd y$. (A possible $(f_1,f_2,f_3)$ is $f_i(x^i,y)=\sin x^i+C_i$ for constant $C_i>2$)
    
    Now we take a perturbation
    \[
        \varphi=-\det\tilde A^{-1}\tilde\alpha^{123}+\det\tilde A^{-1}\tilde A_{ij}\tilde\beta\tilde\alpha^i\tilde\theta^j+\frac{1}{2}\varepsilon_{ijk}\tilde\alpha^i\tilde\theta^{jk}+c\tilde\theta^{123},
    \]
    Let $U_i,\tilde X_i,\tilde Y$ be the dual basis of $\tilde\theta^i,\tilde\alpha^i,\tilde\beta$, then by calculation,
    \begin{align*}
        g_\varphi&=\left(1-\frac{c^2\det\tilde A^{-1}}{4}\right)^{-1/3}(\tilde A_{ij}\tilde\theta^i\tilde\theta^j+\frac{c}2\det\tilde A^{-1}\tilde A_{ij}\tilde\alpha^i\tilde\theta^j+\det\tilde A^{-1}\tilde  A_{ij}\tilde\alpha^i\tilde\alpha^j+\det\tilde A^{-1}\beta^2),\\
        \vol_{g_\varphi}&=\left(1-\frac{c^2\det\tilde A^{-1}}{4}\right)^{1/3}\det\tilde A^{-1}\tilde\alpha^{123}\tilde\beta\tilde\theta^{123}.
    \end{align*}
    To make sure $g_\varphi$ is a Riemannian metric, we need $\det\tilde A>c^2/4$.
    
    Under the new notation,
    \begin{align*}
        A_{ij}&:=g(U_i,U_j)=\left(1-\frac{c^2\det\tilde A^{-1}}{4}\right)^{-1/3}\tilde A_{ij},\\
        a&:=c\det A^{-1/2}=c\det\tilde A^{-1/2}\left(1-\frac{c^2\det\tilde A^{-1}}{4}\right)^{1/2},\\
        \theta^i&:=\tilde\theta^i-\frac{c}2\det\tilde A^{-1}\tilde\alpha^i,\\
        \alpha^i&:=\frac12\varepsilon^{ijk}U_{jk}\lrcorner\varphi=\tilde\alpha^i+c\tilde\theta^i,\\
        \hat\alpha^i&:=\alpha^i-c\theta^i=\left(1-\frac{c^2\det\tilde A^{-1}}{2}\right)\tilde\alpha^i,\\
        \omega_i&:=-cU_i\lrcorner\varphi+\frac12\varepsilon_{ijk}\alpha^{jk}=-c\det\tilde A^{-1}\tilde A_{ij}\tilde\beta\tilde\alpha^j+\frac12\varepsilon_{ijk}\tilde\alpha^{jk}.
    \end{align*}
    So an orbit is associative iff
    \[
        c\det\tilde A^{-1/2}=\sqrt{2}.
    \]
    Also,
    \[
        \rmd\theta^i=-\frac{c}2\rmd(\det\tilde A^{-1})\wedge\tilde\alpha^i.
    \]
    So the perturbation is a non-trivial example.
\end{example}~

\subsubsection{Torsion-free condition}~
Besides $\underline\omega$ singled out from $\varphi$, Lemma \ref{lem:canonical-non-isotropic} also gives the definite triple of $2$-forms 
\begin{align*}
     &\tilde\omega^k
    =-\frac12c\varepsilon^{ijk}U_{ij}\lrcorner{*}\varphi+\alpha^k\beta
\end{align*}
from $*\varphi$, which are now closed since $\alpha^k$'s, $\beta$ and $*\varphi$ are closed. This gives an extra hypersymplectic structure on the $4$-orbifold $X$ and its orbifold-covering $\boldsymbol X$.

\begin{thm}[Torsion-free $\rmG_2$-structure with a non-isotropic $\T^3$-orbit]
\label{thm:nonisotropic-torsionfree-structure}
    Let $\varphi$ be a torsion-free $\rmG_2$-structure on the 7-manifold $M$. Suppose $\varphi$ is invariant under an effective $\mathbb{T}^3$-action, and there is a non-isotropic $\mathbb{T}^3$-orbit. Then besides the conclusions in Theorem \ref{thm:nonisotropic-closed-structure}, we have 
    \begin{align*}
      *\varphi&=-\frac{1}{c^3}\beta\alpha^{123}-\frac{1}{2c^3}\varepsilon_{ijk}\pi^*\tilde\omega^i\wedge\alpha^{jk}-\frac{1}{c^3}\pi^*\omega_i\wedge\pi^*\beta\wedge\alpha^i+\pi^*\bar\Omega,
    \end{align*}
where $\tilde{\underline{\omega}}=\left(\tilde\omega^1,\tilde\omega^2,\tilde\omega^3\right)$ is another hypersymplectic structure on the good 4-orbifold $X$, $\beta$ is a closed 1-form and $\bar\Omega$ is a 4-form on $X$. Moreover, 
\begin{align*}
\boldsymbol\varpi^*\left(*\varphi\right)
=*\boldsymbol\varphi
& =-\frac{1}{c^3}\boldsymbol\beta\boldsymbol\alpha^{123}-\frac{1}{2c^3}\varepsilon_{ijk}\tilde{\boldsymbol\omega}^i\wedge\boldsymbol\alpha^{jk}-\frac{1}{c^3}\boldsymbol\omega_i\wedge\boldsymbol\beta\boldsymbol\alpha^i+\bar{\boldsymbol\Omega},\\
& =
-\boldsymbol\beta\wedge\rmd t^{123}-\frac{1}{2c}\varepsilon_{ijk}\tilde{\boldsymbol\omega}^i\wedge\rmd t^{jk}-\frac{1}{c^2}\boldsymbol\omega_i\wedge\boldsymbol\beta\wedge \rmd t^i+\bar{\boldsymbol\Omega},
\end{align*}
where $\underline{\tilde{\boldsymbol\omega}}=\varpi^*\underline{\tilde\omega}$ is the pull back hypersymplectic structure on $\boldsymbol X$, $\boldsymbol\beta=\varpi^*\beta$ and $\bar{\boldsymbol\Omega}=\varpi^*\bar\Omega$.
\end{thm}

\begin{thm}[Generalized Gibbons-Hawking Ansatz for non-isotropic and non-associative]
\label{thm:Gibbons-Hawking}
    The non-associative part is uniquely determined by the global coframe $\left\{\hat\alpha^1,\hat\alpha^2,\hat\alpha^3,\beta\right\}$ with metric
    \begin{align*}
        g=\frac{\det A^{-1}A_{ij}}{b^2}\hat\alpha^i\hat\alpha^j+\frac{\det A^{-1}}{b^2}\beta^2
    \end{align*}
    on a 4-manifold $\mathcal{X}$, where $A=(A_{ij})$ is a positive-definite $3\times 3$ matrix-valued function such that 
    \begin{enumerate}
        \item $\partial_iA^{ij}=0$ for $j=1,2,3$;
        \item $\rmd\beta=0$ and $\rmd\hat\alpha^k=z^k_i\beta\hat\alpha^i+\varepsilon_{ijk}w^{kl}\hat\alpha^{ij}$ (for $k=1,2,3$), where
    \begin{align*}
        2w^{kl}&=\frac{a^2}{b^2}A^{kp}\partial_qA_{rp}\varepsilon^{qrl}-\frac{c\partial_yA^{kl}}{b^2},\;\; k,l=1,2,3,\\
        z_k^p&=\frac{\partial_q\left(c\det A^{-1}A_{ik}\right)}{b^2}\varepsilon^{qip}-\frac{a^2A_{ik}\partial_yA^{ip}}{b^2},\;\; k,p=1,2,3,
    \end{align*}
    and $\left\{\partial_1,\partial_2,\partial_3,\partial_y\right\}$ is the frame dual to the coframe $\left\{\hat\alpha^1,\hat\alpha^2,\hat\alpha^3,\beta\right\}$.
    \end{enumerate}
\end{thm}
\begin{proof}
    Write\footnote{We should also emphasize that due to the usage of $\alpha^i$ and $\varepsilon_{ijk}$ etc, the $c^{-1}z_i^k$ and $-2c^{-1}w^{kl}$ here corresponds to $z_k^i$ and $w_k^l$ of Madsen-Swann \cite[Section 3.1]{Madsen-Swann}.} $\rmd\hat\alpha^k=:z^k_i\beta\hat\alpha^i+\varepsilon_{ijk}w^{kl}\hat\alpha^{ij}$. The closedness of $\Omega$, whose formula is given in Equation \eqref{Omega-in-nonisotropic-nonassociative}
, implies
    \begin{align*}
        0=c^2\rmd\Omega=\left(\partial_y\left(\frac1{b^2}\right)+\frac{1}{b^2}z_i^i\right)\beta\hat\alpha^{123}.
    \end{align*}
    Then
    \begin{align*}
        \partial_y\left(\frac1{b^2}\right)+\frac{1}{b^2}z_i^i=0.
    \end{align*}
    
    The closedness of $\underline{\omega}$, whose formula is given in Equation \eqref{hypersymplectic-in-nonisotropic-nonassociative}, implies
    \begin{align*}
       0={}\rmd\omega_k={}&\left(\partial_k\left(\frac1{b^2}\right)+\frac{2}{b^2}\varepsilon_{ijk}w^{ij}\right)\hat\alpha^{123}\\
        &+\left(-\frac{1}{2b^2}z_k^i+\frac{c\det A^{-1}}{b^2}A_{pk}w^{pi}+\partial_q\left(\frac{c\det A^{-1}A_{rk}}{2b^2}\right)\varepsilon^{qri}\right)\varepsilon_{ist}\hat\alpha^{st}\beta, \;\; \forall k.
    \end{align*}
    Then the coefficients respectively give
    \begin{align*}
        0&=\partial_k\left(\frac1{b^2}\right)+\frac{2}{b^2}\varepsilon_{ijk}w^{ij},\;\; \forall k, \\
        0&=-\frac{1}{b^2}z_k^i+\frac{2c\det A^{-1}}{b^2}A_{pk}w^{pi}+\partial_q\left(\frac{c\det A^{-1}A_{rk}}{b^2}\right)\varepsilon^{qri}, \;\; \forall k, i.
    \end{align*}
    
    The closedness of $\underline{\tilde\omega}$, whose formula is given in Equation \eqref{dual-hypersymplectic-in-nonisotropic-nonassociative} implies
    \begin{align*}
        0=\rmd\tilde\omega^i={}&\frac{1}{b^2}\partial_pA^{pi}\hat\alpha^{123}\\
        &+\left(-\frac{c}{2b^2}z_p^m+\frac1{b^2}A_{pq}w^{qm}+\partial_k\left(\frac1{2b^2}\right)A_{pq}\varepsilon^{kqm}+\frac{cA_{pq}\partial_yA^{mq}}{2b^2}\right)A^{pi}\varepsilon_{mst}\hat\alpha^{st}\beta, \;\; \forall i.
    \end{align*}
    Then the coefficients respectively give
    \begin{align*}
        0={}&\partial_pA^{pi},\;\;\forall i,\\
        0={}&-\frac{c}{b^2}z_p^m+\frac2{b^2}A_{pq}w^{qm}+\partial_k\left(\frac1{b^2}\right)A_{pq}\varepsilon^{kqm}+\frac{cA_{pq}\partial_yA^{mq}}{b^2}, \;\;\forall p,m.
    \end{align*}
    
    Overall, the torsion-free condition is equivalent to
    \begin{align}
        \label{eq-t2-1}\partial_y\left(\frac1{b^2}\right)+\frac{1}{b^2}z_i^i
        & =0,\\
        \label{eq-t2-2}\partial_k\left(\frac1{b^2}\right)+\frac{2}{b^2}\varepsilon_{ijk}w^{ij}
        & =0,\;\; \forall k,\\
        \label{eq-t2-3} -\frac{1}{b^2}z_k^p+\frac{2c\det A^{-1}}{b^2}A_{ik}w^{ip}+\partial_q\left(\frac{c\det A^{-1}A_{ik}}{b^2}\right)\varepsilon^{qip}
        & =0,\;\; \forall k,p,\\
        \label{eq-t2-4}\partial_j A^{ij}
        & =0,\;\; \forall i,\\
        \label{eq-t2-5}-\frac{c}{b^2}z_k^p+\frac2{b^2}A_{ik}w^{ip}+\partial_q\left(\frac1{b^2}\right)A_{ik}\varepsilon^{qip}+\frac{cA_{ik}\partial_yA^{ip}}{b^2}
        & =0, \;\; \forall k,p.
    \end{align}
    
    For any $k,p$, viewing \eqref{eq-t2-3}, \eqref{eq-t2-5} as a system of linear equations of $z_k^p,A_{ik}w^{ip}$, we could derive
    \begin{align*}
        2w^{ip}&=\frac{a^2}{b^2}A^{ij}\partial_qA_{rj}\varepsilon^{qrp}-\frac{c\partial_yA^{ip}}{b^2},\\
        z_k^p&=\frac{\partial_q(c\det A^{-1}A_{ik})}{b^2}\varepsilon^{qip}-\frac{a^2A_{ik}\partial_yA^{ip}}{b^2}.
    \end{align*}
    
    Also, the equations \eqref{eq-t2-1}, \eqref{eq-t2-2} automatically hold considering the divergence free condition \eqref{eq-t2-4} is satisfied.
\end{proof}

\begin{remark}This theorem is a partial generalization of \emph{Madsen-Swann Ansatz} for toric $\rmG_2$-manifold \cite[Theorem 3.5]{Madsen-Swann}. Actually, $\rmd\rmd\hat\alpha^k=0$ (for $k=1,2,3$) also give a second-order PDE system of $A$. However, the Lie brackets of the derivatives depends on $A$. So it is more complicated, and we leave it for future investigation. 
\end{remark}

\subsection{Two Liouville type theorems}
We are going to prove two Liouville type theorems about complete torsion-free $\rmG_2$-structures invariant under an effective $\T^3$-action, under two extra geometric assumptions about the orbits. Notice that whether an orbit is associative is determined by whether $a=1$, and $a$ is related to the orbit volume through $a=c\det A^{-1}$; hence assuming the orbit volume to be constant forces $a$ to be a constant, which naturally splits the discussion into the non-associative case ($a<1$) and the associative case ($a=1$); a mixture of the two types cannot occur.

\subsubsection{With orbits being non-isotropic non-associative and of constant volume}

\begin{thm}[Complete Torsion-free $\rmG_2$-structure with a non-isotropic non-associative $\T^3$-orbit and constant orbit volume]
\label{thm:complete-torsion-free-constant-volume-non-associative}
    Let $\varphi$ be a complete torsion-free $\rmG_2$-structure on the $7$-manifold $M$. Suppose $\varphi$ is invariant under an effective $\mathbb{T}^3$-action, and there is a non-isotropic non-associative $\mathbb{T}^3$-orbit, and the volume of orbits is constant $1$. Then the action is free, and $\left(M,\varphi\right)$ is flat. Moreover, the $\rmG_2$-structure is
    \begin{align*}
      \varphi=-\frac1{1-a^2}\hat\alpha^{123}+\frac{1}{1-a^2}\beta\hat\alpha^i\theta^i+\frac12{\varepsilon}_{ijk}\hat\alpha^i\theta^{jk}-\frac{a}{2\left(1-a^2\right)}{\varepsilon}_{ijk}\hat\alpha^{ij}\theta^k+a\theta^{123}.
    \end{align*}
    where $a\in(0,1)$ is constant, $\theta^1,\theta^2,\theta^3$ are the connection 1-forms, and $\hat\alpha^1,\hat\alpha^2,\hat\alpha^3,\beta$ is a closed global coframe on the manifold $M/\T^3$.
\end{thm}
\begin{proof}
    Using Theorem \ref{thm:nonisotropic-torsionfree-structure}, we can assume $M=\T^3\times X$ where $X$ is simply-connected, and the action is the canonical $\T^3$-action on the first factor. By the assumptions, $|U_1\wedge U_2\wedge U_3|_{g_\varphi}=\det A^{\frac{1}{2}}=\det B^{-1}$ is constant on the nonempty open set of non-associative orbits. The formula $a=\frac{c}{\sqrt{\det A}}$ implies $a$ is a constant in $(-1,1)$ on this open set, and must continue to be this constant on the boundary of this open set. This implies the orbits on the boundary continue to be non-associative, and thus all orbits are non-associative. Moreover, $a$ is constant on $M$. The existence of a non-isotropic orbit implies all orbits are non-isotropic, and therefore the action is free.  Using \eqref{data-in-alpha-hat-alpha}, \eqref{metric-in-alpha-hat-alpha} we can write
    \begin{align*}
        \varphi&=\frac{1}{c^2}\alpha^{123}-\frac{1}{2b^2c^2}\varepsilon_{ijk}\hat\alpha^{ij}\alpha^k+\frac{\det A^{-1}A_{ij}}{b^2c}\beta\hat\alpha^i\alpha^j+\frac{2}{b^2c^2}\hat\alpha^{123},\\
        g_\varphi&=\frac{A_{ij}}{c^2}\alpha^i\alpha^j-\frac{2A_{ij}}{c^2}\alpha^i\hat\alpha^j+\frac{A_{ij}}{b^2c^2}\hat\alpha^i\hat\alpha^j+\frac{\det A^{-1}}{b^2}\beta^2,
    \end{align*}
    where $a,b,\det A,\det B$ are all constants, and $\left\{\alpha^1,\alpha^2,\alpha^3,\hat\alpha^1,\hat\alpha^2,\hat\alpha^3,\beta\right\}$ is a global coframe on $M$. Let $\left\{X_1,X_2,X_3,Y\right\}$ be the frame dual to the coframe $\left\{\hat\alpha^1,\hat\alpha^2,\hat\alpha^3,\beta\right\}$ on the 4-manifold $M/\T^3$. Define $V=\frac{b}{\det B} Y$ and take $\gamma$ to be an integral curve of $V$ on $M$, then we have $\abs{V}\equiv 1$ and $\nabla_V V=0$. So $\gamma$ is a geodesic. Since $M$ is complete, the definition of $\gamma$ can be extended to $\R$. Also, since $\frac{\det B}{b}\beta$ is closed and $X$ is simply-connected, we have $\frac{\det B}{b}\beta=\rmd y$ for some function $y:X\to\R$. Also, $\rmd(y\circ\pi)=\frac{\det B}{b}\beta$, and $\partial_t\left(y\circ\pi\circ \gamma\right)=\rmd y\left(\dot \gamma\right)=1$. So $\gamma$ is injective, i.e. not a loop.
    
    Take any $t_1<t_2$, and assume a curve $\gamma':[t_1,t_2]\to M$ connecting $\gamma(t_1)$ and $\gamma(t_2)$ minimizes their distance. Note that $y:(M,g_\varphi)\to(\R,\rmd y^2)$ is a Riemannian submersion,
    \begin{align*}
        L(\gamma')\geqslant L(y\circ \gamma')\geqslant L(y\circ \gamma|_{[t_1,t_2]})=L(\gamma|_{[t_1,t_2]}).
    \end{align*}
    So $\gamma|_{[t_1,t_2]}$ itself minimizes the distance. Therefore $\gamma$ is a geodesic line.
    
    By Cheeger-Gromoll's Splitting Theorem in Riemann geometry and the fact that $\Ric\; g_\varphi=0$ due to the torsion-freeness of $\varphi$, $(M,g_\varphi)=(\R,\beta^2)\times(N,\bar g)$, where $N$ is a complete 6-manifold. Actually, $N$ can be viewed as one level set of $y$, and flow generated by the vector field $V$ induces the isometric splitting. 
    
    Since $\T^3$ preserves $\beta$ and $\Iso(\R,\beta^2)=\R\rtimes\Z_2$, the $\T^3$-action is trivial on $(\R,\beta^2)$. Then the $\T^3$-action is effective on $(N,g_N)$. So $A_{ij}=g(U_i,U_j)$ is independent of $y$, i.e. $\partial_yA_{ij}=0$. Also, on $M/\T^3=\R\times N/\T^3$, $A_{ij}\hat\alpha^i\hat\alpha^j$ is independent of $y$. That is,
    \begin{align*}
        0=L_Y\bar g&=\partial_yA_{ij}\hat\alpha^i\otimes\hat\alpha^j+A_{ij}L_Y\hat\alpha^i\otimes\hat\alpha^j+A_{ij}\hat\alpha^i\otimes L_Y\hat\alpha^j\\
        &=A_{ij}z^i_k\hat\alpha^k\otimes\hat\alpha^j+A_{ij}z^j_k\hat\alpha^i\otimes\hat\alpha^k\\
        &=\left(A_{pj}z^p_i+A_{pi}z^p_j\right)\hat\alpha^i\otimes\hat\alpha^j
    \end{align*}
    Then $\forall i, j,$
    \begin{align*}
        0&=A_{pj}z^p_i+A_{pi}z^p_j\\
        &=A_{kj}\left(\frac{a^2\partial_pA_{qi}}{b^2c}\varepsilon^{pqk}-\frac{a^2A_{qi}\partial_yA^{qk}}{b^2}\right)+A_{ki}\left(\frac{a^2\partial_pA_{qj}}{b^2c}\varepsilon^{pqk}-\frac{a^2A_{qj}\partial_yA^{qk}}{b^2}\right)\\
        &=A_{kj}\frac{a^2\partial_pA_{qi}}{b^2c}\varepsilon^{pqk}+A_{ki}\frac{a^2\partial_pA_{qj}}{b^2c}\varepsilon^{pqk}\\
        &=^*\frac{2a^2}{b^2c}A_{kj}\partial_pA_{qi}\varepsilon^{pqk}
    \end{align*}
    where $*$ is from the fact that $\forall l$,
    \begin{align*}
   A_{ki}\partial_pA_{qj}\varepsilon^{pqk}\varepsilon^{ijl}&=\det A\cdot\partial_pA_{qj}\left(A^{pj}A^{ql}-A^{pl}A^{qj}\right)\\
        &=-\det A\cdot\partial_pA^{pl}-\det A\cdot A^{pl}\partial_p\log\det A\\
        &=0,
    \end{align*}
    where in the last equality we use the divergence free condition $\partial_p A^{pl}=0$ and the condition that $\det A$ is constant. Finally, item (2) of Theorem \ref{thm:Gibbons-Hawking} implies $z^i_j=0$ and $w^{kl}=0$, and then $0=\rmd\hat\alpha^k$, i.e. the connection $\underline\theta$ on the $\T^3$-bundle is flat.\\
    
    The hypersymplectic structures $\underline\omega$ and $\underline{\tilde\omega}$ on $X$ obtained in Theorem \ref{thm:nonisotropic-torsionfree-structure} satisfy the relation 
    \begin{align}
    \tilde\omega^i-cA^{ij}\omega_j+\beta\hat\alpha^i
    &=0
    \end{align}
    by Lemma \ref{lem:relation-between-two-hypersymplectic}. The flatness of the connection $\underline\theta=\left(\theta^1,\theta^2,\theta^2\right)$ implies $\rmd \hat\alpha^i=-c\rmd\theta^i=0$, and therefore 
    \begin{align*}
        \rmd\left(A^{ij}\omega_j\right)=0. 
    \end{align*}
    On the other hand, since the $Q$-matrix (recall $Q_{ij}=\frac{1}{2}\langle \omega_i, \omega_j\rangle_{g_{\underline\omega}}$) for the hypersymplectic structure $\underline\omega$ is $Q_{ij}=\frac{A_{ij}}{\det A^\frac{1}{3}}$ by item (3) of Lemma \ref{lem:canonical-non-isotropic} and $\det A$ is constant, we have 
    \begin{align*}
        \rmd\left(Q^{ij}\omega_j\right)=0
    \end{align*}
    which precisely means that $\underline\omega$ is torsion-free.
    
    Note that the Riemannian product $M=\R\times N$ gives
    \begin{align*}
        (X,g_{sub})=(\R,\rmd y^2)\times (N/\T^3,\frac{\det A^{-1}A_{ij}}{b^2}\hat\alpha^i\hat\alpha^j),
    \end{align*}
    we have that
    \begin{align*}
        g_{\underline{\omega}}=\det A^{2/3}a^2\rmd y^2+\det A^{2/3}\cdot\frac{\det A^{-1}A_{ij}}{b^2}\hat\alpha^i\hat\alpha^j.
    \end{align*}
    Since $\det A^{2/3},\det A^{2/3}a^2$ are both constant, $g_{\underline{\omega}}$ is a complete metric. We can use \cite[Theorem 25]{Fine-Yao2}, which states that ``\emph{any torsion-free hypersymplectic structure which defines a complete Riemannian metric is hyperk\"ahler}'', to conclude that $\left(\boldsymbol X,\underline{\boldsymbol\omega}\right)$ is hyperk\"ahler, i.e. $A$ is a constant matrix. Therefore $M$ is flat.
\end{proof}

\subsubsection{With all orbits being associative}\label{subsect:associative}\,
In this subsection we make the special assumption that all $\T^3$-orbits are associative (which in particular implies the existence of a non-isotropic orbit), equivalently $\Omega=0$ in the decomposition of $\varphi$ in Theorem \ref{thm:nonisotropic-closed-structure}. 

To fix the scale, let $\T^3=S^1\times S^1\times S^1=\left\{\left(e^{it_1}, e^{it_2}, e^{it_3}\right)|t_1, t_2, t_3\in \mathbb{R}\right\}$, and using $U_i=\frac{\partial}{\partial t^i}$ as the generating vector field for the $\T^3$-action on $M$. Then $c=\varphi\left(U_1,U_2, U_3\right)=\varphi\left(\frac{\partial}{\partial t^1},\frac{\partial}{\partial t^2},\frac{\partial}{\partial t^3}\right)$ and $\alpha^i=\frac{1}{2}\varepsilon^{ijk}U_{jk}\lrcorner \varphi=c\rmd t^i$.
Using Theorem \ref{thm:nonisotropic-torsionfree-structure}, we can consider the pull-back complete torsion-free $\rmG_2$-structure 
\begin{align*}
    \boldsymbol \varphi
    = 
    c\rmd t^{123}-\frac{1}{c}\rmd t^i\wedge \boldsymbol\omega_i
\end{align*}
on $\T^3\times \boldsymbol X$ which equivariantly covers $M$ by the map $\boldsymbol\varpi$, and $\varpi:\boldsymbol X\rightarrow X$ is an orbifold-covering. The volume of any fiber $\mathbb{T}^3$ under the metric $g_{\boldsymbol \varphi}$ is $\left(2\pi\right)^3c$. Up to a global rescaling, we can simply assume $c=1$, and arrive at precisely the $\rmG_2$-structure considered in \cite{Fine-Yao}. Since $pr: \T^3\times \boldsymbol X\rightarrow \boldsymbol X$ is a Riemannian submersion from $g_{\boldsymbol \varphi}=A_{ij}\rmd t^i\rmd t^j\oplus g_{\underline{\boldsymbol \omega}}$ to $g_{\underline{\boldsymbol\omega}}$, the completeness of $g_{\boldsymbol \varphi}$ implies the completeness of $g_{\underline{\boldsymbol\omega}}$. The torsion-freeness of $\boldsymbol \varphi$ implies the torsion-freeness of $\underline{\boldsymbol\omega}$, and therefore we can again use \cite[Theorem 25]{Fine-Yao2} to conclude $\left(\boldsymbol X,\underline{\boldsymbol\omega}\right)$ is hyperk\"ahler, i.e. $A$ is a constant matrix.~

\begin{thm}[Complete Torsion-free $\rmG_2$-structure with all $\T^3$-orbits associative]
\label{thm:associative-hyperkahler}
    Let $\varphi$ be a complete torsion-free $\rmG_2$-structure on the $7$-manifold $M$. Suppose $\varphi$ is invariant under an effective $\mathbb{T}^3$-action with all orbits being associative. Then the action is almost-free, and there exists a complete hyperk\"ahler $4$-manifold $\left(\boldsymbol X, \underline{\boldsymbol \omega}\right)$ and a $\T^3$-equivariant covering $\boldsymbol \varpi:\T^3\times \boldsymbol X\rightarrow M$, such that 
    \begin{align*}
        \boldsymbol\varpi^*\varphi
        = 
        c\rmd t^{123}
        - 
        \frac{1}{c}\rmd t^1\wedge \boldsymbol\omega_1
        - 
        \frac{1}{c}\rmd t^2\wedge \boldsymbol\omega_2 
        - 
        \frac{1}{c}\rmd t^3\wedge \boldsymbol\omega_3
    \end{align*}
    where $\left(2\pi\right)^3c=\text{Vol}_\varphi(\mathcal{O})$ is the volume of any principal $\T^3$-orbit\footnote{The volume of an exceptional orbit $\mathcal{O}_p$ is $\frac{\left(2\pi\right)^3c}{|H_p|}$.}. Moreover, $X=M/\T^3$ is a good orbifold with cyclic isolated singularities equipped with an orbifold hyperk\"ahler structure $\underline\omega=\left(\omega_1,\omega_2,\omega_3\right)$, and there exists an orbifold-covering $\varpi:\boldsymbol X\rightarrow X$ such that $\varpi^*\omega_i=\boldsymbol\omega_i$ for $i=1,2,3$.
\end{thm}


\begin{example}
        Consider $M=\mathbb{R}^4\times \left(\left[0,2\pi\right]\times \mathbb{T}^2\right)/\sim$ where ``$\sim$'' refers to the equivalence relation
    \[
    \left(\mathbf x, \left(0,t^2,t^3\right)\right)
    \sim 
    \left(-\mathbf x, \left(2\pi, \theta^2,\theta^3\right)\right), \; \; \mathbf x\in \mathbb{R}^4, \; \left(t^2,t^3\right)\in \mathbb{T}^2.
    \]
    Consider 
    \[
    \varphi
    =
    \rmd t^{123}-\rmd t^1\wedge\omega_1-\rmd t^2\wedge\omega_2-\rmd t^3\wedge\omega_3
    \]
    where $\left(\omega_1,\omega_2,\omega_3\right)$ is the standard hyperk\"ahler triple on $\mathbb{R}^4$. The group $\mathbb{T}^3=\left([0,4\pi]/{\sim}\right)\times\T^2$ acts as standard \emph{translations} in the $\mathbb{T}^3$-factor, preserving $\varphi$: 
    \begin{itemize}
      \item The orbit $\left\{\mathbf 0\right\}\times \mathbb{T}^3$ is exceptional;
      \item The orbit of $\left(\mathbf{x},\underline t\right)$ for $\mathbf x\neq \mathbf 0$ is principal.
    \end{itemize}
The orbit space $\left(\mathbb{R}^4/\mathbb{Z}_2, \underline\omega\right)$ is the standard flat hyperk\"ahler orbifold.
\end{example}

    \begin{remark}
        It should be noticed that $i\left(\tilde X\right)$ is a locally embedded coassociative submanifold of $\left(M,\varphi\right)$ by the mechanism of Lemma 2.13 of \cite{Kari-Lotay}, where hypersymplectic geometry also plays an crucial rule. The difference is that in our current situation, the hypersymplectic metric on the coassociative submanifold coincides with the induced Riemannian metric $g_\varphi$ on $i\left(\tilde X\right)$ while it is not the case for all the examples considered in \cite{Kari-Lotay} (see Remark 2.16 therein).
    \end{remark}

\section{Isotropic orbit}
\label{sect:isotropic}
The terminology ``isotropic orbit'' corresponds to the condition $\varphi(U_1,U_2,U_3)=0$ for any choice of basis of generating vector fields for the $\mathbb{T}^3$-action, and this is independent of the choice of basis of Lie algebra of $\mathbb{T}^3$. As mentioned above, it is proved in \cite[Lemma 2.6]{Madsen-Swann}), there exists no exceptional orbits even though there might exist singular orbits.

Concrete examples include the flat models 
\begin{itemize}
    \item $M=\mathbb{T}^2\times \mathbb{R}\times \mathbb{C}^2$ 
    with coordinate $x,y,u,z,w$, and 
    \[
    \varphi
    = 
    \rmd u\wedge\rmd x\wedge\rmd y 
    - 
    \rmd u\wedge\frac{i}{2}\left(\rmd z\wedge\rmd \bar z+ \rmd w\wedge\rmd \bar w\right) 
    - 
    \text{Re}\left( \left(\rmd x - i\rmd y\right)\wedge \rmd z\wedge \rmd w\right), 
    \]
    \item $M=S^1\times \mathbb{C}^3$ 
    with coordinate $x,z_1, z_2, z_3$, and
    \[
    \varphi= 
    \frac{i}{2}\rmd x \wedge \left(\rmd z_1\wedge\rmd \bar z_1 + \rmd z_2\wedge \rmd \bar z_2 + \rmd z_3\wedge \rmd \bar z_3\right)
    + 
    \text{Re}\left(\rmd z_1\wedge\rmd z_2\wedge \rmd z_3\right),
    \]
\end{itemize}
and non-flat models $S^1\times T^*S^3$ coming from the Stenzel Calabi-Yau structure \cite{St, Madsen-Swann} and Bryant-Salamon structures on the Spinor bundle over $S^3$ \cite{BS}. Many more examples, including infinite families, can be constructed via the methods of \cite{FHN} and \cite{DMS}.

\subsection{Hypersymplectic foliation}\label{sect:hypersymplectic-foliation}\,
According to Lemma \ref{lem:canonical-decomposition-isotropic}, on the open sets $M_0$ of principal orbits we have the decomposition
\begin{align}
\label{phi-isotropic-deomposition}
    \varphi
    & = 
    -\det A^{-1}\alpha^{123} + \alpha^i\wedge\omega_i, \\
    *\varphi
    & = 
    -\beta\theta^{123} + \frac{1}{2}\varepsilon_{ijk}\alpha^{ij}\wedge\tilde\omega^k, \nonumber
\end{align}
where the triple of $2$-forms
\begin{align}
    \omega^i 
    & = \frac{1}{2}\varepsilon_{ijk}\theta^{jk} - \left(\det A^{-1}\right)A_{ip}\beta\theta^p, \\
    \tilde\omega^i
    & = 
    -\frac{1}{2}\varepsilon_{pqr}A^{ip}\theta^{qr}
    + 
    \left(\det A^{-1}\right) \beta\theta^i,
\end{align}
where the $1$-forms $\alpha^i:=\frac12\varepsilon^{ijk}U_{jk}\lrcorner\varphi$ and $\beta:=U_{123}\lrcorner{*\varphi}$, the $1$-forms $\theta^i$'s for the connection $\mathbb{U}^{\perp_{g_\varphi}}=\bigcap_i \ker\theta^i$ on the principal $\mathbb{T}^3$-bundles $\pi: M_0\to X_0$ , are introduced in the beginning to section \ref{sect:non-isotropic}. We could not conclude $\rmd \theta^i\equiv 0$ and we could not conclude $\det B$ to be constant neither in general. Actually, in all the examples of toric $\rmG_2$-manifold considered in \cite{Madsen-Swann}, even the flat models, the connections $\theta^i$'s are not flat, singular orbits do appear and $\det B$ blows up near the singular orbits.\\ 

Since $\alpha_1, \alpha_2, \alpha_3$ are closed $1$-forms and are linearly independent everywhere on $M_0$, the distribution $\mathbb{K}=\bigcap_i\ker\alpha^i$ determines a regular foliation whose leaves are $4$-dimensional submanifolds of $M_0$. The distribution is independent of the choice of generating vector fields $U_1,U_2,U_3$. Let $Y_0$ be the \emph{space of leaves}. Since the 1-forms $\alpha_i$'s restricts to $0$ on each $\mathbb{T}^3$-orbit, the $\mathbb{T}^3$-action restricts to free action on $N_y$ for each $y\in Y_0$. We will see in the following that $N_y$ is inheriting some special geometries from the background $\rmG_2$-structure and the $\mathbb{T}^3$-action.\\

The induced Riemannian metric on $M_0$ is 
\begin{align*}
    g_\varphi&=A_{ij}\theta^i\theta^j+\det A^{-1}A_{ij}\alpha^i\alpha^j
    +A^{-1}\beta^2.
\end{align*}

Let $\left\{\mathfrak{n}_1, \mathfrak{n}_2, \mathfrak{n_3}, \mathfrak{m}, U_1, U_2, U_3\right\}$ be the frame dual to the coframe $\left\{\alpha^1, \alpha^2,\alpha^3, \beta, \theta^1, \theta^2, \theta^3\right\}$ on $M_0$, then 
\begin{align*}
\widehat\omega_i^y
:=\omega^i|_{N_y}
= \left(\mathfrak{n}_i\lrcorner\varphi\right)|_{N_y}
= \frac{1}{2}\varepsilon_{ijk}\widehat\theta^{jk}
- 
\left(\det A^{-1}\right) A_{ip}\widehat\beta\widehat\theta^p, \;\; i=1,2,3
\end{align*}
satisfies 
\begin{equation}
\widehat\omega_i^y\wedge\widehat\omega_j^y
=
2\det A^{-1}A_{ij} \widehat\theta^{123}\widehat\beta
\end{equation}
where $\widehat\theta^i=\theta^i|_{N_y}, \widehat\beta=\beta|_{N_y}$. Since the curvature $2$-forms $\rmd\theta^i$'s are linear combination of $\alpha^{pq}$ and $\beta\alpha^r$ on $M_0$, we have 
\[
\rmd\widehat\theta^i\equiv 0
\]
on $N_y$. The $\mathbb{T}^3$ invariance of $\beta=\left(*\varphi\right)\left(U_1, U_2, U_3, \cdot\right)$ implies $\mathcal{L}_{U_i}\beta=0$, and thus 
\begin{align*}
    &\rmd \beta \left(\mathfrak{m}, U_i\right)
    = 
    \mathfrak{m}\left(\beta\left(U_i\right)\right)-U_i\left(\beta\left(\mathfrak{m}\right)\right) + \beta\left(\mathcal{L}_{U_i}\mathfrak{m}\right)
    =
    - \left(\mathcal{L}_{U_i}\beta\right)\left(\mathfrak{m}\right)
    =
    0.\\
    &\rmd \beta \left(U_i, U_j\right) 
    = 
    U_i\left(\beta\left(U_i\right)\right) 
    - 
    U_j\left(\beta\left(U_i\right)\right) 
    - 
    \beta\left( \left[U_i,U_j\right]\right) 
    =
    0.
\end{align*}
It follows that $\rmd \beta$ is also a linear combination of $\alpha^{pq}$ and $\beta\alpha^r$ as $\rmd \theta^i$'s above. The term $\rmd \left(\det A^{-1}A_{ip}\right)$ is a linear combination of $\alpha^i$'s and $\beta$, implying that only the components involving $\alpha^i$'s survive in the computations of $\rmd \omega^i$'s\footnote{As a contrast, the components involving $\beta$ in the computation of $\rmd \tilde\omega^i$ also survives.}.\\

Therefore, $\widehat\omega_i^y$ is closed $2$-form on $N_y$ for $i=1,2,3$ and 
$\underline{\widehat\omega}^y=\left(\widehat\omega_1^y,\widehat\omega_2^y, \widehat\omega_3^y\right)$ is a hypersymplectic structure on $N_y$. The metric $\left(N_y, \underline\omega^y\right)$ may not coincide with $g_\varphi|_{N_y}$ as remarked in \cite[Remark 2.16]{Kari-Lotay}. It follows from the formula \eqref{phi-isotropic-deomposition} that 
\begin{align}
\varphi|_{N_y}
& = 0, \;\; \\
*\varphi|_{N_y}
& = \text{dvol}_{N_y}, 
\end{align}
meaning that $N_y$ is ``formally calibrated'' by $*\varphi$. In the special case that $*\varphi$ is closed, $N_y$ is indeed a co-associative submanifold as noted by Madsen-Swann(see the paragraph immediately after \cite[Proposition 3.2]{Madsen-Swann}). This also follows from the general discussion of coassociative fibrations of Karigiannis-Lotay (see \cite[Corollary 2.15]{Kari-Lotay}). We should notice that $N_y$ is not co-associative in case $*\varphi$ is not closed.\\

If we choose another basis of generating vector fields of the action $U_i'=K_i^{\phantom{i}j}U_j$, then 
\[
{\alpha'^i} = \left(\det K\right)\left(K^{-1}\right)_p^{\phantom{p}i}\alpha^p, \;\;
\beta'  = \left(\det K\right)\beta, \;\;
\theta'^i= \left(K^{-1}\right)_l^{\phantom{l}i}\theta^l
\]
and 
\[
\mathfrak{n}_i' = \left(\det K\right)^{-1}K_i^{\phantom{i}p}\mathfrak{n}_p,\;\;
\mathfrak{m}' = \left(\det K\right)^{-1}\mathfrak{m}, \;\;
U_i' = K_i^{\phantom{i}p}U_p.
\]
As a consequence, 
\[
\widehat\omega_i'^{y}
= 
\left(\mathfrak{n}_i'\lrcorner \varphi\right)|_{N_y}
= 
\left(\det K\right)^{-1}K_i^{\phantom{i}p}\widehat\omega_p^y, \;\; 
g_{\underline{\widehat\omega}'}= g_{\underline{\widehat\omega}},
\]
i.e. the hypersymplectic geometry inherited on $N_y$ is independent of the particular choice of generating vector fields. 

\subsection{Trivalent graph}\label{sect:trivalent-graph}\,

If $\varphi$ is only assumed to be closed, the multi-moment maps $\nu_1,\nu_2,\nu_3$ for $\varphi$ exist locally. More precisely, we have 
\begin{lemma}
\label{lemma:local-multimoment}
    For any orbit $\mathcal{O}$, there exists a $\mathbb{T}^3$-invariant open neighborhood $\mathcal{U}$ of $\mathcal{O}$, and smooth $\mathbb{T}^3$-invariant functions $\nu_1, \nu_2, \nu_3$ defined on $\mathcal{U}$ such that 
    \[
    \rmd\nu_i=\frac12\varepsilon_i^{\phantom{i}jk}U_j\wedge U_k\lrcorner\varphi, \;\; i=1,2,3.
    \]   
    The functions $\nu_i$'s satisfying that $\nu_i|_\mathcal{O}=0$ are uniquely determined by the basis $U_1, U_2, U_3$.
\end{lemma}

\begin{proof}
    Take any $p\in \mathcal{O}$, and metric ball $\mathbb{B}_p$ (centered at $0$) in $\left(T_p \mathcal{O}\right)^\perp$ such that $\exp_p: \mathbb{B}_p\to M$ is injective. Translating $\mathbb{B}_p$ using elements of $\mathbb{T}^3$, we get a $\mathbb{T}^3$-equivariant diffeomorphism 
    \[
    \exp_\mathcal{O}^\perp: \mathbb{B}\left( \subset N\mathcal{O}\right)\rightarrow M
    \]
    where $\mathbb{B}$ is the disc bundle (of fixed radius with respect to $g_\varphi$) in the normal bundle of $\mathcal{O}$.

    The form $\alpha_i=U_{jk}\lrcorner \varphi$ restricts to closed $1$-from on the submanifold $\exp_p\mathbb{B}_p$, and therefore it is the differential of a function $\nu_i$ defined on this submanifold, and it is uniquely determined if we further require $\nu_i(p)=0$. The isotropy group $H_p\subset \mathbb{T}^3$ acts on $\exp_p\mathbb{B}_p$ in a way which preserves the restriction of the $\alpha_i$, and therefore $\nu_i$ is $H_p$-invariant. By translating $\exp_p \mathbb{B}_p$ to $\exp_q\mathbb{B}_q$ for any $q\in \mathbb{T}^3\cdot p=\mathcal{O}$, we obtain naturally the function $\nu_i$ defined on $\mathcal{U}=\exp_\mathcal{O}^\perp\left(\mathbb{B}\right)$ that is the potential for $\alpha_i$.
\end{proof}

Since $\beta:=U_1\wedge U_2\wedge U_3\lrcorner{*}\varphi$ may not be closed, there may not exist multi-moment map for $*\varphi$,  Nevertheless, most of singularity analysis in \cite[Section 4.2]{Madsen-Swann} still holds. Different from the torsion-free case, certain higher order covariant derivatives of $\nu_1, \nu_2, \nu_3$ and $\beta$, which contain terms of the form $(\nabla^a\varphi)(\nabla^bU_i,\nabla^cU_j,\cdot)_p$, may not agree with the flat model.

Explicitly, for $p\in M$ with $H_p\simeq\T^2$, take $U_1,U_2,U_3$ corresponding to the flat model defined in \cite[Section 4.1]{Madsen-Swann}, then
\begin{align*}
    &(\nabla\nu_1)_p
     =(\nabla\nu_2)_p=(\nabla\nu_3)_p=\beta_p=0,\\
    & (\nabla^2\nu_2)_p
     =\varphi(\nabla U_3,U_1,\cdot)_p\neq0,\;\;
    (\nabla^2\nu_3)_p
     =\varphi(U_1,\nabla U_2,\cdot)_p\neq0,\;\;
    (\nabla^2\nu_1)_p
     =(\nabla\beta)_p=0,\\
    & (\nabla^3\nu_1)_p
     =2\varphi(\nabla U_2,\nabla U_3,\cdot)_p\neq0,\;\;
    (\nabla^2\beta)_p
     =2\left({*}\varphi\right)(U_1,\nabla U_2,\nabla U_3,\cdot)_p
    \neq0.
\end{align*}
These agree with the flat model $S^1\times \mathbb{C}^3$ as in \cite[Section 4.2.2]{Madsen-Swann}, while 
\begin{align*}
    \left(\nabla^3 \nu_2\right)_p
    = 
    2\left(\nabla\varphi\right)\left(\nabla U_3, U_1, \cdot\right)_p,\;\;
    \left(\nabla^3 \nu_3\right)_p
    = 
    2\left(\nabla\varphi\right)\left(U_1, \nabla U_2,\cdot\right)_p
\end{align*}
may not agree with the flat model. This means, at a point $p$ with stabilizer $\mathbb{T}^2$, $\nu_2$ and $\nu_3$ agree with the flat model to order $2$, $\nu_1$ agrees with the flat model to order $3$ while $\beta$ agrees the flat model to order $2$. 

For $p\in M$ with $H_p\simeq S^1$, take $U_1,U_2,U_3$ corresponding to the flat model, then
\begin{align*}
    &(\nabla \nu_1)_p=(\nabla \nu_2)_p=\beta_p=0,\;\;
    (\nabla \nu_3)_p=\varphi(U_1,U_2,\cdot)_p\neq0,\\
    &(\nabla^2\nu_1)_p=\varphi(U_2,\nabla U_3,\cdot)_p\neq0,\;\;
    (\nabla^2\nu_2)_p=\varphi(\nabla U_3,U_1,\cdot)_p\neq0,\\
    &(\nabla\beta)_p=\left(*\varphi\right)(U_1,U_2,\nabla U_3,\cdot)_p\neq0.
\end{align*}
They agree with the flat model $\mathbb{T}^2\times \mathbb{R}\times \mathbb{C}^2$ as in \cite[Section 4.2.3]{Madsen-Swann}, while 
\begin{align*}
    &\left(\nabla^2 \nu_3\right)_p 
    = 
    \left(\nabla\varphi\right)\left(U_1,U_2, \cdot\right)_p 
    + 
    \varphi\left(\nabla U_1, U_2, \cdot\right)_p 
    + 
    \varphi\left(U_1, \nabla U_2, \cdot\right)_p, \\
    &\left(\nabla^2\beta\right)_p 
    = 
    \left(\nabla \left(*\varphi\right)\right)\left(U_1, U_2, \nabla U_3, \cdot\right)_p 
    + 
    2\left(*\varphi\right)\left(\nabla U_1, U_2, \nabla U_3, \cdot\right)_p 
    + 
    2\left(*\varphi\right)\left(U_1, \nabla U_2, \nabla U_3, \cdot\right)_p
\end{align*}
may not agree with the flat model. This means, at a point $p$ with stabilizer $S^1$, $\nu_1$ and $\nu_2$ agree with the flat model to order $2$, while $\nu_3$ and $\beta$ agree with the flat model to order $1$. 

The above singularity analysis shows similarly that
\begin{prop}
    For any orbit $\mathcal{O}$, there exists a $\mathbb{T}^3$-invariant open neighborhood $\mathcal{U}$ of $\mathcal{O}$ and basis $U_1,U_2,U_3$ of the generating vector fields, such that under the map 
    \[
    {\boldsymbol{\nu}}=\left(\nu_1,\nu_2,\nu_3\right): \mathcal{U}\rightarrow \mathbb{D}\subset \mathbb{R}^3,
    \]
    the image of singular orbits 
    \begin{itemize}
        \item corresponding to one-dim stabilizers is $\mathbb{D}\bigcap\left\{\nu_1=\nu_2=0\right\}$;
        \item corresponding to two-dim stabilizers is 
        $\left\{(0,\nu_2,0)\in \mathbb{D}|\nu_2>0\right\}\cup \left\{(0,0,\nu_3)\in \mathbb{D}|\nu_3<0\right\}\cup \left\{(0,-\nu,\nu)\in \mathbb{D}|\nu>0\right\}\cup \left\{\mathbf{0}\right\}$.
    \end{itemize}
    \end{prop}

\subsection{Another Liouville type theorem}
If we try to extend the hypersymplectic foliation structure to the whole of $M$, we will run into the new phenomenon of certain singular hypersymplectic foliation and singular hypersymplectic structures. It is interesting to see how the structure of the space of leaves and those hypersymplectic structures incoorporate to a $\mathbb{T}^3$-invariant closed $\rmG_2$-structure with isotropic orbits. Nevertheless, under the same geometric assumption of constant volume of orbits, we can obtain the following Liouville type theorem.

\begin{thm}[Complete Torsion-free $\rmG_2$-structure with an isotropic $\T^3$-orbit and constant orbit volume]
\label{thm:complete-isotropic-const-volume}
    Let $\varphi$ be a complete torsion-free $\rmG_2$-structure on the $7$-manifold $M$. Suppose $\varphi$ is invariant under an effective $\mathbb{T}^3$-action, and there is an isotropic $\mathbb{T}^3$-orbit, and the volume of orbits is constant. Then the action is free, and $\left(M,\varphi\right)$ is flat. Moreover, the $\rmG_2$-structure is
    \begin{align*}
      \varphi=\beta\alpha^i\theta^i+\frac12{\varepsilon}_{ijk}\alpha^i\theta^{jk}-\alpha^{123}.
    \end{align*}
    where $\theta^1,\theta^2,\theta^3$ are the connection 1-forms, and $\alpha^1,\alpha^2,\alpha^3,\beta$ is a closed global coframe on the manifold $M/\T^3$.
\end{thm}
\begin{proof}
    Since the volume of orbits is constant, we have $\det A$ is constant, where $A_{ij}:=g_\varphi(U_i,U_j)$. So there is no singular orbits, we can assume the structures can be rewritten as
    \begin{align*}
        \varphi=\det A^{-1}A_{ij}\beta\alpha^i\theta^j+\frac12\varepsilon_{ijk}\alpha^i\theta^{jk}-\det A^{-1}\alpha^{123},\\
        g_\varphi=A_{ij}\theta^i\theta^j+\det A^{-1}A_{ij}\alpha^i\alpha^j+\det A^{-1}\beta^2.
    \end{align*}
    Using Theorem \ref{thm:nonisotropic-torsionfree-structure}, we can assume $M=\T^3\times X$ where $X$ is simply-connected, and the action is the canonical $\T^3$-action on the first factor. Since $\alpha^1,\alpha^2,\alpha^3,\beta$ are closed, there are functions $\nu^1,\nu^2,\nu^3,\mu$ such that $\alpha^i=\rmd\nu^i,\beta=\rmd\mu$. By \cite{Madsen-Swann}, the torsion-free condition is equivalent to
    \begin{align}
        \partial_iA^{ij}
        & =0,\;\;\forall j,\\
     \partial_\mu^2A^{ij}+A^{pq}\partial_p\partial_qA^{ij}-\partial_pA^{iq}\partial_qA^{jp}
        &=0,\;\;\forall i,j,
    \end{align}
    where $\partial_\mu=\partial/\partial\mu,\partial_p=\partial/\partial\nu^p$. Then
    \begin{align*}
        0&=A_{ij}\left(\partial_\mu^2A^{ij}+A^{pq}\partial_p\partial_qA^{ij}-\partial_pA^{iq}\partial_qA^{jp}\right)\\
        &=\partial_\mu(A_{ij}\partial_\mu A^{ij})-\partial_\mu A_{ij}\partial_\mu A^{ij}+A^{pq}\partial_p(A_{ij}\partial_qA^{ij})-A^{pq}\partial_pA_{ij}\partial_qA^{ij}-A_{ij}\partial_pA^{iq}\partial_qA^{jp}\\
        &=A^{ik}A^{jl}\partial_\mu A_{ij}\partial_\mu A_{kl}+A^{pq}A^{ik}A^{jl}\partial_pA_{ij}\partial_qA_{kl}-A^{pk}A^{iq}A^{jl}\partial_pA_{ij}\partial_qA_{kl}\\
        &=\abs{\partial_\mu A}_A^2+\frac12\abs{T_A}_{g_A}^2
    \end{align*}
    where $T_A=(\partial_pA_{ij}-\partial_iA_{pj})\rmd \nu^p\otimes \rmd \nu^i\otimes \rmd \nu^j$ is a $(0,3)$-tensor on $M/\T^3$, and $|\cdot|_A$ measures its norm with respect to the metric $g_A:=\rmd\mu^2+A_{ij}\rmd \nu^i\otimes \nu^j$, and $|P|_A^2=A^{ik}P_{ij}A^{jl}P_{kl}$ for symmetric $3\times 3$-matrix $P$ measures the square norm in the symmetric space $\Sym_+\left(3\times 3\right)$ (as in \cite{Fine-Yao}).
    So we have
    \begin{align*}
        \partial_\mu A_{ij}
        & =0,\;\;\forall i,j,\\
       \partial_pA_{ij}-\partial_iA_{pj}
       & =0, \;\;  \forall i,j,p.
    \end{align*}
    Now define
    \begin{align*}
        \hat\omega_i:=\det A^{-1/2}A_{ij}\alpha^j\beta+\frac12\varepsilon_{ijk}\alpha^{jk},
    \end{align*}
    then
    \begin{align*}
        \hat\omega_i\wedge\hat\omega_j
         =2\det A^{-1/2}A_{ij}\alpha^{123}\beta, & \;\; \forall i,j,\\
        \rmd\hat\omega_i
         =\det A^{-1/2}\partial_pA_{ij}\alpha^{pj}\beta
        =0, &\;\; \forall i,\\
        \rmd(A^{ij}\hat\omega_i)=\frac{1}{2}\varepsilon_{ipq}\partial_rA^{ij}\alpha^{rpq}=\partial_rA^{rj}\alpha^{123}
         =0, &\;\; \forall j.
    \end{align*}
    So $\underline{\hat\omega}$ is a torsion-free hypersymplectic structure on $M/\T^3$. Moreover, its induced metric is a multiple of $g_{sub}$ and therefore complete. Applying \cite[Theorem 25]{Fine-Yao2}, we have that $(X,\underline{\hat\omega})$ is a hyperk\"ahler 4-manifold and $A$ is constant. So $M$ is flat. Moreover, under suitable basis $\left\{U_1, U_2, U_3\right\}$, $A\equiv id$ and we obtain the form for $\varphi$ claimed in the theorem.
\end{proof}

\noindent {\bfseries{Acknowledgement}}: We would like to thank Yohsuke Imagi for very helpful discussions. This work was supported by the \emph{National Key R\&D Program of China [2025YFA1018200 to C.Y.]}. The article is part of the master thesis of Z.-Y. Zhou.\\

\bibliographystyle{amsplain}
\bibliography{ref}{}

\end{document}